\documentclass[reqno]{amsart}

\usepackage [mathscr]{eucal}
\usepackage{amsmath,amsthm,amssymb,amscd}
\usepackage{amsrefs}
\usepackage{epsf}
\usepackage{amsfonts}
\usepackage{dsfont}
\usepackage{latexsym}
\usepackage{mathrsfs}
\usepackage{layout}

\newtheorem{theorem}{Theorem}[section]
\newtheorem{lemma}[theorem]{Lemma}
\newtheorem{proposition}[theorem]{Proposition}
\newtheorem{corollary}[theorem]{Corollary}

\numberwithin{equation}{section}

\DeclareMathOperator*{\essosc}{ess\ osc}
\DeclareMathOperator*{\esssup}{ess\ sup}
\DeclareMathOperator*{\essinf}{ess\ inf}
\DeclareMathOperator*{\di}{div} 
\DeclareMathOperator{\dist}{dist}
\DeclareMathOperator*{\data}{data}
\DeclareMathOperator{\diam}{diam}

\begin{document}

\title[H\"{o}lder continuity]{H\"{o}lder continuity of a bounded weak solution of generalized parabolic $p-$Laplacian equations}
\author[S. Hwang]{Sukjung Hwang}
\thanks{This work is based on Ph.D thesis  at Iowa State University. The author was partially supported by EPSRC EP/J017450/1 grant at the University of Edinburgh.} 
\address{The School of Mathematics ,The University of Edinburgh and Maxwell Institute of Mathematical Sciences, EH8 9TL Edinburgh, UK}
\email{s.hwang@ed.ac.uk}

\author[G. Lieberman]{Gary M. Lieberman}
\address{Department of Mathematics , Iowa State University, Ames, IA 50011}
\email{lieb@iastate.edu}


\begin{abstract}
Here we generalize quasilinear parabolic $p-$Laplacian type equations to obtain the prototype equation as
\[
u_t - \text{div} \left(\frac{g(|Du|)}{|Du|} Du\right) = 0,
\]
where a nonnegative, increasing, and continuous function $g$ trapped in between two power functions $|Du|^{g_0 -1}$ and $|Du|^{g_1 -1}$ with $1<g_0 \leq g_1 < \infty$. Through this generalization in the setting from Orlicz spaces, we provide a uniform proof with a single geometric setting that a bounded weak solution is locally H\"{o}lder continuous without separating degenerate and singular types. By using geometric characters, our proof does not rely on any of alternatives which is based on the size of solutions.
\end{abstract}

\maketitle

\section*{Introduction}\label{I}

In 1957, DeGiorgi \cite {DG57} showed that bounded weak solutions of linear elliptic partial differential equations are H\"older continuous, and his method was used by Ladyzhenskaya and Ural$'$tseva in \cite {LaUr61} to show that bounded weak solutions of the quasilinear elliptic equation
\[
\di \mathbf A(x,u,Du)=0
\]
 are H\"older continuous if there are positive constants $p>1$, $C_0$, and $C_1$ such that
\[
 \mathbf A(x,u,\xi)\cdot\xi \ge C_0|\xi|^p, |\mathbf A(x,u,\xi)| \le C_1|\xi|^{p-1}
 \]
 for all $\xi\in\mathbb R^N$, where $N$ is the number of space dimensions.
 (The theorem of De Giorgi is really just the case $p=2$ here.)
 For parabolic equations
 \begin {equation} \label {I:gen}
 u_t-\di \mathbf A(x,t,u,Du) = 0,
 \end {equation}
 Ladyzhenskaya and Ural$'$tseva followed De Giorgi's method with some modifications but they were only able to prove H\"older continuity under the structure conditions
  \begin {equation} \label {EAp}
\mathbf A(x,t,u,\xi)\cdot\xi \ge C_0|\xi|^p, |\mathbf A(x,u,\xi)| \le C_1|\xi|^{p-1}
 \end {equation}
 when $p=2$.

 There was little progress on the H\"older continuity of solutions when $p\neq2$ until 1986, when DiBenedetto \cite {DB86} proved the H\"older continuity result for $p>2$.
 A key new step was his introduction of the concept of intrinsic scaling (first introduced in a simpler setting in \cite {DB86j}), which has since become an important aspect in the theory and which is discussed at great length in \cite {Urb08}.
 It took several more years until the joint work of Chen and DiBenedetto  \citelist {\cite {ChDB88} \cite {ChDB92}} showed that bounded weak solutions are H\"older continuous also for $p<2$.
 Unfortunately, these proofs are quite technical and their exposition (for example, in Chapters III and IV of \cite {DB93}) is quite long.
 More recently, Gianazza, Surnachev, and Vespri \cite {GiSuVe10} developed a more geometric approach to the H\"older continuity of solutions to equations when $p>2$; their proof is simpler and more natural than the original one, but several issues from that proof still remain that we address here.

 The more important ones are related to the distinction between the cases $p>2$ and $p<2$.
 All previously published proofs of H\"older continuity have treated this cases separately because of different qualitative behavior of solutions in the two cases.
 For example, any nonnegative solution of \eqref {I:gen} which vanishes at a point $(x_0,t_0)$ also vanishes in any cylinder with top center point $(x_0,t_0)$ if $p\ge2$; however, when $p<2$, nonnegative solutions  generally become zero in finite time.
(We refer the reader to Sections VI.2, VII.2, and VII.3 of \cite {DB93} for a more complete discussion of these phenomena.)
Such behavior must be accounted for, but our proof points out some significant common elements.
A further issue is that the newer proofs (see the Remark on page 278 of \cite {GiSuVe10} for the case $p>2$ and Section 4 of \cite {DBGiVe10} for a related result in case $p<2$) give a H\"older exponent which degenerates as $p$ approaches $2$; in both cases, the proof must be further modified for $p$ close to $2$ if the H\"older exponent is remain positive near $p=2$ even though the original proofs of H\"older continuity (in \citelist {\cite {DB86} \cite {ChDB88} \cite {ChDB92}}) allowed a stable H\"older exponent in this case.

In this paper, we take a more general approach to the problem: We study \eqref {I:gen} when there is an increasing function $g$ such that
\begin{subequations} \label{I:gen-str}
\begin{gather}
\mathbf A(x,t,u,\xi)\cdot \xi \ge C_0G(|\xi|),\label{I:gen-str1} \\
|\mathbf A(x,t,u,\xi)| \le C_1g(|\xi|)\label{I:gen-str2}
\end{gather}
\end{subequations}
for some positive constants $C_0$ and $C_1$, where $G$ is defined by
\[
G(\sigma)=\int_0^\sigma g(s)\, ds,
\]
and we assume that there are constants $g_0$ and $g_1$ satisfying $1<g_0\le g_1 <\infty$ such that
\begin {equation} \label{I:DeltaNabla}
g_0G(\sigma) \le \sigma g(\sigma) \le g_1G(\sigma)
\end {equation}
for all $\sigma>0$.
(The two inequalities in \eqref{I:DeltaNabla} are essentially $\Delta_2$ and $\nabla_2$ conditions in Orlicz space theory as in Section I.3 \& I.4 of \cite{KrRu61} and in Section 2.3 of \cite{RaRe91}. The precise connection between \eqref {I:DeltaNabla} and the $\Delta_2$ and $\nabla_2$ conditions is the topic of \cite {Lie04}.)
We further restrict $G$ by requiring that $g_1\le2$ (the singular case) or $g_0\ge2$ (the degenerate case).
The structure \eqref {EAp} is contained in this model as the special case $g(s)=s^{p-1}$, in which case we may take $g_0=g_1=p$, and $p=2$ will fit into both the degenerate and the singular structures studied here.
In addition, our structure allows consideration of more general equations; as shown on pages 313 and 314 of \cite{Lie91}, for any $\alpha$ and $\beta$ with $1<\alpha<\beta<\infty$, we can find a function $g$ satisfying \eqref {I:DeltaNabla} such that
\[
\limsup_{s\to\infty} \frac {g(s)}{s^\beta}>0,\quad \liminf_{s\to\infty} \frac {g(s)}{s^\alpha}<\infty,
\]
so we consider a class of structure functions $g$ much wider than that of just power functions.
In this way, we obtain a uniform proof of H\"older continuity (with appropriate uniformity of constants) for all $p\in(1,\infty)$ at once under the structure condition \eqref {EAp} as well as a proof of H\"older continuity under more general structure conditions.
We note especially that the exponent is uniformly controlled (assuming \eqref {EAp}) over any finite range of $p$ that stays away from $1$, so our result in this case is stable as $p$ approaches $2$.
We also point out that if we replace $G$ and $g$ by suitable multiples of these functions (and appropriately modifying $C_0$ and $C_1$), we can achieve any number of normalizations: for example $G(1)=1$, $g(1)=1$, $C_0=1$, $C_1=1$. It's interesting to note that our estimates are (mostly) independent of the normalization.

We point out here that the motivation for considering \eqref {I:DeltaNabla} comes from \cite {Lie91} in which corresponding results for elliptic equations were proved.
The extension of the methods used in \cite {Lie91} for proving H\"older continuity of weak solutions to parabolic equations is not straightforward, and this paper presents the only such extension known to the authors.

For the extension, we also need a suitable definition of weak solution, and we present it here.
For an arbitrary open set $\Omega\subset \mathbb R^{n+1}$, we introduce the generalized Sobolev space $W^{1,G}(\Omega)$, which consists of all functions $u$ defined on $\Omega$ with weak derivative $Du$ satisfying
\[
\iint_{\Omega} G(|Du|)\, dx\,dt <\infty.
\]
We say that $u\in C_{\text{loc}}(\Omega) \cap W^{1,G}(\Omega)$ is a weak subsolution of \eqref{I:gen} if
\[
0 \geq -\iint_{\Omega} u\varphi_t\, dx\, dt + \iint_{\Omega} \textbf{A}(x,t,u,Du)\cdot D\varphi\, dx\, dt
\]
for all $\varphi\in C^1(\bar\Omega)$ which vanish on the parabolic boundary of $\Omega$; a weak supersolution is defined by reversing the inequality.
In fact, we shall use a larger class of $\varphi$'s which we discuss in a later section.

Our method of proof also uses some recent ideas of Gianazza, Surnachev, and Vespri \cite{GiSuVe10}, who gave a different proof for the H\"older continuity in \cites{Chen84,DB86}.
While \cites{Chen84,DB86} examine an alternative based on the size of the set on which $|u|$ is close to its maximum, the method in \cite{GiSuVe10} use a geometric approach from regularity theory and Harnack estimates.
We shall not discuss the Harnack estimate here, but the geometry from \cite{GiSuVe10} is an important ingredient of our proof.
On the other hand, \cite{GiSuVe10} takes advantage of the nonvanishing of nonnegative solutions of degenerate equations for all time, so we need to uses some ideas from \cites{ChDB88, ChDB92} to analyze the corresponding behavior of more general equations.
In fact, our proof of H\"older continuity for degenerate equations is based very strongly on that in \cite {DB86}, and we shall rely on the alternative method mentioned above in this case. Our proof of H\"older continuity for singular equations is based very strongly on that in \cite {ChDB92} but we introduce two new ideas to simplify the proof. One new idea shows how to prove the result directly for $p\ge2$ (if we restrict to a power function for $g$) and the other idea shows how to obtain an expansion of positivity result similar to Proposition 4.5.1 in \cite {DBGiVe12}.

We begin by discussing the two alternatives and expansion of positivity results.  These results all assume that $u$ is a nonnegative weak supersolution of \eqref {I:gen} in a scaled cylinder.  The first alternative states that, if $u$ is large on most of one subcylinder in a suitable family of subcylinders of the original cylinder, then $u$ is bounded from zero on all of a subcylinder with the same center-top point as the original cylinder.  The second alternative states that, if $u$ is large on a fixed fraction of every subcylinder in this family of subcylinders, then $u$ is bounded from zero on all of a subcylinder with the same center-top point as the original cylinder. 
The expansion of positivity states that if $u$ is large on half of the original cylinder, then it is bounded away from zero on all of a subcylinder with the same center-top point as the original cylinder. 
Eventually, we shall see the precise quantitative description of these results.

In Section~\ref {S1}, we provide some preliminary results, mostly involving notation for our geometric setting.  In Section~\ref {S2}, we use the supersolution results to show that if the oscillation of a solution $u$ of \eqref {I:gen} over a cylinder is less than or equal to a number $\omega$ appropriately connected to the cylinder, then the oscillation over a smaller subcylinder is less than or equal to $\sigma\omega$ with $\sigma\in(0,1)$;
this oscillation control is then used to prove H\"older continuity. 
The supersolution estimates are proved in Section~\ref  {S3}, based on some integral inequalities which are proved in Section~\ref {S4}.

\section{Preliminaries}\label{S1}

    \subsection{Notation}

    \begin{itemize}
    \item[(1)] The set of parameters $\{g_0, g_1, N, C_0, C_1\}$ are the data.

    \item[(2)] Let $K_{\rho}^{y}$ denote the $N-$dimensional cube centered at $y\in \mathbb{R}^{N}$ with the side length $2\rho$, i.e.,
    \[
    K_{\rho}^{y} := \left\{ x\in \mathbb{R}^{N} : \max_{1 \leq i \leq N} |x_{i} - y_{i}| < \rho \right\}.
    \]
    For simpler notation, let $K_{\rho} := K_{\rho}^{0}$.

    \item[(3)] For given $(x_0, t_0) \in \mathbb{R}^{N+1}$, and given positive constants $\theta$, $\rho$ and $k$, we say
    \begin{gather*}
    T_{k, \rho}(\theta) := \theta k^2 G\left(\frac{k}{\rho}\right)^{-1}, \\
    Q_{k, \rho}^{x_0, t_0}(\theta) := K_{\rho}^{x_0} \times [t_0 -T_{k,\rho}, t_0], \\
    Q_{k, \rho}(\theta) := Q_{k, \rho}^{0,0}(\theta).
    \end{gather*}
We also abbreviate
\[
T_{k,\rho}= T_{k,\rho}(1), \quad Q_{k,\rho}^{x_0,t_0}  =  Q_{k, \rho}^{x_0, t_0}(1), \quad Q_{k\rho}=Q_{k,\rho}(1).
\]
  \end{itemize}

    \subsection{Geometry}

    The local energy estimate \eqref{S4:LocalE} plays a crucial role in this paper which is nonhomogeneous unless $g_0 = g_1 = 2$. By controlling the length of time axis, we make two competing terms in \eqref{S4:LocalE} equivalent; that is, find $T_{k,\rho}$ from
    \[
    G^{r-1} \left(\frac{\omega}{\rho}\right) \omega^{s+2} \frac{1}{T_{k,\rho}} \sim  G^{r} \left(\frac{\omega}{\rho}\right) \omega^{s},
    \]
    for some constants $r$ and $s$ which directly leads to our definition of $T_{k,\rho}$.

    This idea is so called intrinsic scaling introduced by DiBenedetto \cites{DB86j,DB93,Urb08}; roughly speaking, a weak solution of parabolic $p-$Laplacian type equation behaves like a solution of the heat equation in an intrinsically scaled cylinder. To reflect different natures of degenerate and singular equations, original proof  by DiBenedetto \cite{DB86} and Chen and DiBenedetto \cites{ChDB88, ChDB92} applied intrinsic time scaling for degenerate equations ($p>2$) and intrinsic side length scaling for singular equations ($1<p<2$), respectively. A key part of our argument is that the time scaling is also appropriate for singular equations.

The parameter $\theta$ is introduced to simplify some arguments.  It should be noted that the arguments in \cite {DB93} also introduce various similar constants.
    Now, suppose that $u$ is a bounded weak solution of \eqref{I:gen} under \eqref{I:gen-str} in some open subset $\Omega$ of $\mathbb R^{N+1}$ and let $(x_0,t_0)\in \Omega$. Since $\Omega$ is open, there are positive constants $r$ and $s$ such that $K_r^{x_0}\times (t_0-s,t_0) \subset \Omega$. If we set
\[
    R=  \min \left\{ r, \frac{\omega}{G^{-1} \left(\theta \omega^2 s^{-1} \right)} \right\},
\]
we conclude that
\[
    Q_{\omega, 4R}^{x_0, t_0} \subset \Omega.
     \]
    Without loss of generality, we let $(x_0, t_0) = (0,0)$. For any arbitrary given cylinder, we can fit  the cylinder $Q_{\omega, 4R}$ in $\Omega$ by selecting $R$ properly. Basically, we are going to work with the cylinder $Q_{\omega, 4R}$ to find a proper subcylinder where a solution has less oscillation eventually leading to H\"{o}lder continuity.

    \subsection{Useful inequalities}

    Because of the generalized functions $g$ and $G$, we are not able to apply H\"{o}lder inequality or typical Young's inequality. Here we deliver essential inequalities which will be used through out the paper.

    \begin{lemma}\label{S1:ineq}
    For a nonnegative and nondecreasing function $g \in C[0,\infty)$, let $G$ be the antiderivative of $g$. Suppose that $g$ and $G$ satisfies (\ref{I:DeltaNabla}). Then for all nonnegative real numbers $\sigma$, $\sigma_1$, and $\sigma_2$, we have
    \begin{enumerate}
    \item[(a)] $G(\sigma)/\sigma$ is a monotone increasing function.
    \item[(b)] For $\beta > 1$,
        \[
        \beta^{g_0} G(\sigma) \leq G(\beta \sigma) \leq \beta^{g_1} G(\sigma).
        \]
    \item[(c)] For $0< \beta < 1$,
        \begin{equation}\label{S1:ineq(c)}
        \beta^{g_1} G(\sigma) \leq G(\beta \sigma) \leq \beta^{g_0} G(\sigma).
        \end{equation}
    \item[(d)]
        $\displaystyle
        \sigma_1 g(\sigma_2) \leq \sigma_1 g(\sigma_1) + \sigma_2 g(\sigma_2).
        $
        \smallskip

    \item[(e)] \emph{(Young's inequality)} For any $\epsilon > 0$,
        \[
        \sigma_1 g(\sigma_2) \leq \epsilon^{1-g_1} g_1 G(\sigma_1) + \epsilon g_1 G(\sigma_2).
        \]
    \end{enumerate}
    \end{lemma}

    \begin{proof}
    This lemma is quoted directly or modified from the Lemma 1.1 from \cite{Lie91}.
    \begin{enumerate}
    \item[(a)] For $\sigma > 0$, due to the left hand side inequality of (\ref{I:DeltaNabla}), we easily obtain
    \[
    \frac{d}{d \sigma}\left( \frac{G(\sigma)}{\sigma} \right) = \frac{\sigma g(\sigma) - G(\sigma)}{\sigma^2}\geq (g_0 -1)\frac{G(\sigma)}{\sigma^2} > 0
    \]
    because $g_0 > 1$.

    \item[(b)] The left inequality of (\ref{I:DeltaNabla}) gives
    \[
    \frac{g_0}{\xi} \leq \frac{g(\xi)}{G(\xi)} \ \text{for } \ \xi \in (0,\infty).
    \]
    By taking the integral over $\sigma$ to $\beta\sigma$, we obtain
    \[
    g_0 \log \frac{\beta \sigma}{\sigma} \leq \log \frac{G(\beta \sigma)}{G(\sigma)}
    \]
    which implies
    \[
    \beta^{g_0} G(\sigma) \leq G(\beta \sigma).
    \]
    Similar argument with the right hand side of (\ref{I:DeltaNabla}) completes the proof.

    \item[(c)]
    Like the proof for (b), but take integrals over $\beta \sigma$ to $\sigma$.

    \item[(d)]
    It is clear because $g$ is nondecreasing function so either
    \[
    \sigma_1 g(\sigma_2) \leq \sigma_1 g(\sigma_1)\quad \text{or} \quad \sigma_1 g(\sigma_2) \leq \sigma_2 g(\sigma_2).
    \]

    \item[(e)]
    For any $0< \epsilon < 1$, because of (d) we obtain
    \[
    \sigma_1 g(\sigma_2) = \epsilon \frac{\sigma_1}{\epsilon} g(\sigma_2) \leq \epsilon \left[\frac{\sigma_1}{\epsilon} g\left(\frac{\sigma_1}{\epsilon}\right) + \sigma_2 g(\sigma_2) \right].
    \]
    Applying the right inequality of (\ref{I:DeltaNabla}) and (b) leads to
    \[
    \sigma_1 g(\sigma_2) \leq \epsilon \left[g_1 G\left(\frac{\sigma_1}{\epsilon}\right) + g_1 G(\sigma_2) \right] \leq \epsilon g_1 \epsilon^{-g_1}G(\sigma_1) + \epsilon g_1 G(\sigma_2).
    \]
    \end{enumerate}
    \end{proof}

    The below inequalities will be used to derive the logarithmic energy estimate \eqref{S4:LogE} which plays a crucial role in Proposition~\ref{S3:prop2}.

    \begin{lemma}\label{S1:H-ineq}
    For any $\sigma >0$, let
    \[ 
    h(\sigma) = \frac{1}{\sigma} \int_{0}^{\sigma} g(s) \,ds,
    \] 
    and
    \[ 
    H(\sigma) = \int_{0}^{\sigma} h(s) \,ds.
    \] 
    Then we have
    \begin{enumerate}
    \item[(a)]
    $\displaystyle
    g_0 h(\sigma) \leq g(\sigma) \leq g_1 h(\sigma).
    $
    \smallskip

    \item[(b)]
    $\displaystyle
    g_0 H(\sigma) \leq G(\sigma) \leq g_1 H(\sigma).
    $
    \smallskip

    \item[(c)]
    $\displaystyle
    (g_0 - 1) h(\sigma) \leq \sigma h'(\sigma) \leq (g_1 - 1) h(\sigma).
    $
    \smallskip

    \item[(d)]
    $\displaystyle
    \frac{1}{g_1} \sigma h(\sigma) \leq H(\sigma) \leq \frac{1}{g_0} \sigma h(\sigma).
    $
    \smallskip

    \item[(e)] For a constant $\beta > 1$,
    \[
    \beta^{g_0} H(\sigma) \leq H(\beta \sigma) \leq \beta^{g_1} H(\sigma)
    \]
    \end{enumerate}
    \end{lemma}

    \begin{proof} Here we note that $h$ acts like $g$ and $H$ acts like $G$.
    \begin{enumerate}
    \item[(a)] Dividing \eqref{I:DeltaNabla} by $\sigma$ complete the proof.
    \item[(b)] Taking integrals to (a) gives the inequality.
    \item[(c)]
    Since
    \[
    h'(\sigma) = \frac{g(\sigma)}{\sigma} - \frac{G(\sigma)}{\sigma^2},
    \]
    applying (\ref{I:DeltaNabla}) completes the proof.

    \item[(d)]
   This inequality follows from (b) since $G(\sigma) = \sigma h(\sigma)$.  \item[(e)] Similar to the proof for (b) on Lemma~\ref{S1:ineq}.
    \end{enumerate}
    \end{proof}
    Our next result concerns some inequalities about integration of a function over various intervals.  We shall use these inequalities in the proof of Lemma~\ref {S3:L7.5}.
    
 \begin {lemma} \label {S1:fintegral}
 Let $f$ be a continuous, decreasing, positive function defined on $(0,\infty)$.  Then, for all $\delta$ and $\sigma\in(0,1)$, we have
\begin {equation} \label {S1:fintegral:E2}
\int_0^1 f(\delta+s)\,ds \le \frac 1\sigma\int_0^{\sigma}f(\delta+s)\,ds.
\end {equation}
If, in addition, for all $\beta>1$ and $\sigma>0$, we have
 \begin {equation} \label {S1:fintegral:E}
\beta f(\beta\sigma) \ge f(\sigma), \quad f(\beta\sigma) \le f(\sigma),
\end {equation}
then, for all $\delta\in (0,1)$, we have
\begin {equation} \label {S1:fintegral:E1}
\int_0^\delta f(\delta+s)\, ds \le \frac 2{\ln (1/\delta)+2} \int_0^1 f(\delta+s)\, ds.
\end {equation}
\end {lemma}
\begin {proof}
To prove \eqref {S1:fintegral:E2}, we define the function $F$ by
\[
F(\sigma) = \sigma \int_0^1 f(\delta+s)\, ds - \int_0^{\sigma} f(\delta+s)\, ds.
\]
Since
\[
F'(\sigma) = \int_0^1f(\delta+s)\, ds - f(\delta+\sigma),
\]
and $f$ is decreasing, it follows that $F'$ is increasing so $F$ is convex. Moreover
\[
F(0)=F(1)=0,
\]
so $F(\sigma) \le0$ for all $\sigma\in(0,1)$.  Simple algebra then yields  \eqref {S1:fintegral:E2}.

To prove \eqref {S1:fintegral:E1}, we first use a change of variables to see that, for any $j\ge1$, we have
\begin {align*}
\int_{j\delta}^{2j\delta} f(\delta+s)\, ds &=\int_0^{j\delta} f((j+1)\delta +\sigma)\, d\sigma \\
&= j\int_0^{\delta} f((j+1)\delta+js))\, ds.
\end {align*}
Since $(j+1)\delta+js\le (j+1)(\delta+s)$ and $f$ is decreasing, we have
\[
\int_{j\delta}^{2j\delta} f(\delta+s)\, ds \ge j \int_0^\delta f((j+1)(\delta+s))\,ds
\]
and then \eqref {S1:fintegral:E} gives
\[
\int_{j\delta}^{2j\delta} f(\delta+s)\, ds \ge \frac j{j+1}\int_0^{\delta} f(\delta+s)\, ds.
\]
We now let $J$ be the unique positive integer such that $2^{-J}<\delta \le 2^{1-J}$ and we take $j=2^i$ with $i=0,\dots,J-1$. Since $j/(j+1)\ge 1/2$, it follows that
\[
\int_0^\delta f(\delta+s)\, ds \le 2\int_{2^i\delta}^{2^{i+1}\delta} f(\delta+s)\,ds.
\]
Since
\[
\int_0^{2^J\delta}f(\delta+s)\, ds = \int_0^\delta f(\delta+s)\, ds + \sum_{i=0}^{J-1} \int_{2^i\delta}^{2^{i+1}\delta} f(\delta+s)\,ds,
\]
we infer that
\[
\int_0^{2^J\delta} f(\delta+s)\, ds \ge [1+\frac 12J]\int_0^\delta f(\delta+s)\, ds.
\] 
The proof is completed by noting that $J>\ln(1/\delta)$ and that
\[
\int_0^1 f(\delta+s)\, ds \ge\int_0^{2^J\delta} f(\delta+s)\, ds.
\]
\end {proof}

Note that condition \eqref {S1:fintegral:E} is satisfied if $f(\sigma)=\sigma^{-p}$ with $0\le p \le 1$, in which case this lemma can be proved by computing the integrals directly.

\section {The basic results and the proof of H\"older continuity} \label {S2}

In this section, we prove the H\"older continuity of solutions of \eqref {I:gen} for singular (that is, equations with $g_1\le 2$) and for degenerate equations (that is, equations with $g_0\ge 2$).  Our proof is based on some estimates for nonnegative supersolutions of the equation, and these estimates will be proved in the next section.  

Our first lemma states that a nonnegative supersolution $u$ of a singular equation is strictly positive in a subcylinder if $u$ is near to the maximum value in more than a half of cylinder.

\begin{lemma} [Main Lemma] \label{S2:MainLemma} 
Let $\omega$ and $R$ be positive constants.  Then there are positive constants $\delta$ and $\mu$, both less than one and determined only by the data such that, if $u$ is a nonnegative solution of \eqref {I:gen} in  
\[
Q =Q_{\delta\omega,2R}\left(\frac 34\right)
\]
with $g_1\le2$ and 
\begin{equation}\label{S2:MainLemma-hyp}
\left| Q\cap \left\{ u \le \frac{\omega}{2}\right\}\right| \le \frac{1}{2}\left| Q\right|,
\end{equation}
then
\begin {equation} \label {S2:MainLemma:estimate}
\essinf_{\mathcal Q} u \geq \mu \omega,
\end {equation}
with
\[
\mathcal Q= Q_{4\mu\omega,R/2}\left( \frac 1{64}\right).
\]
\end{lemma}

We shall prove this lemma in the next section.

For degenerate equations, we need two cases which are usually described as the first alternative and the second alternative.  

For notational convenience, we take $\nu_0$ to be the constant from Proposition~\ref {S3:prop4} corresponding to $\theta=1$ and, with $\omega$ and $R$ given positive constants, we set
\[
\Delta= \omega^2G\left(\frac {\omega}{2R}\right)^{-1}.
\]

Our first alternative is that, if $u$ is a positive subsolution $u$ of a degenerate equation which stays close to its maximum on most of one suitable small subcylinder,  then $u$ is bounded away from zero on a suitable subcylinder.

\begin {lemma}[The first alternative] \label {S5:L1} Let $\theta_0>1$ be a given constant and suppose $u$ is a nonnegative supersolution of \eqref {I:gen} in
\begin {equation} \label {S5:L1:Q}
Q= K_{2R}\times (-\theta_0\Delta,0)
\end {equation}
with $g_0\ge 2$.  If there is a constant $T_0 \in[-\theta_0\Delta,-\Delta]$ such that
\begin {equation} \label {S5:L1:alt1}
|K_{2R}\times(T_0,T_0+\Delta)\cap \left\{u<\frac{\omega}2\right\}| \le \nu_0|K_{2R}|\Delta,
\end {equation}
then there is a constant $\delta_1\in (0,1)$ determined only by $\theta_0$ and data such that
\[
\essinf_{\mathcal Q} u\ge \delta_1 \omega
\]
with
\begin {equation} \label {S5:L1:mQ}
\mathcal Q = Q_{\omega/2,R/2}\left(\frac 12\right).
\end {equation}
\end {lemma}

The proof of this lemma will be given in the next section.

Our second alternative states that if $u$ is a positive subsolution of a degenerate equation which stays close to its maximum on a suitable fraction  of all suitable small subcylinders,  then $u$ is bounded away from zero on a suitable subcylinder.

\begin {lemma}[The second alternative] \label {S5:L2} There are constants $\theta_0>1$ and $\delta_2\in(0,1)$ (determined only by data) such that, if $u$ is a nonnegative supersolution of \eqref {I:gen}in $Q$ (given by \eqref {S5:L1:Q}) with $g_0\ge2$ and
\[
\left|K_{2R}\times (T_0,T_0+\Delta)\cap \left\{u<\frac{\omega}2\right\}\right| \le (1-\nu_0)|K_{2R}|\Delta
\]
for all $T_0\in[-\theta_0\Delta,-\Delta]$, then there is a constant $\delta_2\in (0,1)$, determined only by data, such that
\begin {equation} \label {S5:L2:eqn}
\essinf_{\mathcal Q} u \ge \delta_2\omega
\end {equation}
with $\mathcal Q$ given by \eqref {S5:L1:mQ}.
\end {lemma}
Again, we prove this lemma in the next section.

From these lemmata, we infer a decay estimate for the oscillation of a bounded solution of \eqref {I:gen}.  This lemma is the first place where we can state the result uniformly for singular and for degenerate equations although the proof is different for these two cases.

\begin {lemma} \label {S2:Lmainboth}
Let $C_0$, $C_1$, $g_0$, $g_1$, $\rho$, and $\omega$ be positive constants with $C_0\le C_1$ and $1<g_0\le g_1$.  Suppose also that $g_1\le 2$ or that $g_0\ge 2$ and suppose that $u$ is a bounded weak solution of \eqref {I:gen} in $Q_{\omega,\rho}$ (with $\theta=1$) with 
\[
\essosc_{Q_{\omega,\rho}} u \le \omega.
\]
Then there are positive constants $\sigma$ and $\lambda$, both less than one and determined only by data such that
\begin {equation} \label {S2:Lmainbothest}
\essosc_{Q_{\sigma\omega,\lambda\rho}} u \le \sigma\omega.
\end {equation}
\end {lemma}
\begin {proof}
We first prove the lemma for singular equations, so that $g_1\le2$.
In this case, we take $\delta$ and $\mu$ to be the constants from Lemma~\ref {S2:MainLemma} and we set 
\[
\sigma=1-\mu, \quad \lambda = 4^{-5/g_0}\mu^{(2-g_0)/g_0}.
\]
We also introduce the functions $u_1$ and $u_2$ by
\begin {equation} \label {S2:u12}
u_1=u- \essinf_{Q_{\omega,\rho}} u, \quad u_2=\omega-u_1.
\end {equation}
It follows from Lemma~\ref {S1:ineq} that 
\[
3\left( \frac {\delta\omega}2\right)^2 G\left(\frac {\delta\omega}\rho\right)^{-1} \le \omega^2  G\left(\frac {\omega}\rho\right)^{-1},
\]
and hence the cylinder $Q$ from Lemma~\ref {S2:MainLemma} is a subset of $Q_{\omega,\rho}$ provided $R=\rho/2$.

There are now two cases.  First, if
\[
\left|Q\cap \left\{u_1\le  \frac \omega2\right\}\right| \le \frac 12 |Q|,
\]
then we apply Lemma~\ref {S2:MainLemma} to infer that
\[
\essinf_{\mathcal Q} u_1\ge \mu\omega.
\]
Since
\[
\esssup_{\mathcal Q} u_1\le \omega,
\]
it follows that
\[
\essosc_{\mathcal Q} u = \essosc_{\mathcal Q} u_1\le (1-\mu)\omega=\sigma\omega.
\]
On the other hand if
\[
\left|Q\cap \left\{u_1\le  \frac \omega2\right\}\right| \ge \frac 12 |Q|,
\]
then
\[
\left|Q\cap \left\{u_2\le  \frac \omega2\right\}\right| \le \frac 12 |Q|,
\]
and an application of Lemma~\ref {S2:MainLemma} to $u_2$ implies once again that
\[
\essosc_{\mathcal Q} u \le \sigma \omega.
\]
We now infer from Lemma~\ref {S1:ineq} that
\[
(\sigma\omega)^2G\left( \frac {\sigma\omega}{\lambda\rho}\right)^{-1} \le \left(\frac {2\mu\omega}4\right)^2G\left( \frac {16\mu\omega}{\rho}\right)^{-1}.
\]
Since $\lambda\le\frac 14$, it follows that $Q_{\sigma\omega,\lambda\rho}$ is a subset of the cylinder $\mathcal Q$ from Lemma~\ref {S2:MainLemma}, and \eqref {S2:Lmainbothest} follows.

For the degenerate case, we proceed only a little differently because of the two alternatives.  In particular, we need to introduce some more constants.  This time, we take $\nu_0$ from Proposition~\ref {S3:prop4} (corresponding to $\theta=1$) and then $\theta_0$ from Lemma~\ref {S5:L2}.  We also set $\Delta= \omega^2G(\omega/\rho)^{-1}$, $R= \theta_0^{-1/2}\rho$, $\lambda= 1/(4\theta_0^{1/2})$, and $\sigma=1-\min\{\delta_1,\delta_2\}$, with $\delta_1$ the constant from Lemma~\ref {S5:L1} and $\delta_2$ the constant from Lemma~\ref {S5:L2}. Finally, we define (as in the singular case)  $u_1$ and $u_2$ by \eqref {S2:u12}. With $Q$ defined by \eqref {S5:L1:Q}, we have $Q \subset Q_{\omega,\rho}$.

If there is a $T_0\in(-\theta_0\Delta,-\Delta)$ such that
\[
\left|K_{\rho}\times(T_0,T_0+\Delta)\cap \left\{u_1\le \frac \omega2\right\}\right| \le \nu_0|K_\rho|\Delta,
\]
then Lemma~\ref {S5:L1} applied to $u_1$ implies that
\[
\essinf_{\mathcal Q} u_1 \ge \delta_1\omega
\]
with $\mathcal Q$ defined by \eqref {S5:L1:mQ}.
and hence
\begin {equation} \label {S2:Lmainbothoss}
\essosc_{\mathcal Q} u \le \sigma\omega.
\end {equation}

On the other hand, if
\[
\left|K_{\rho}\times(T_0,T_0+\Delta)\cap \left\{u_1\le \frac \omega2\right\}\right| \ge \nu_0|K_\rho|\Delta
\]
for all $T_0 \in (-\theta_0\Delta,-\Delta)$, it follows that
\[
\left|K_{\rho}\times(T_0,T_0+\Delta)\cap \left\{u_2\le \frac \omega2\right\}\right| \le (1-\nu_0)|K_\rho|\Delta,
\]
so Lemma~\ref {S5:L2} applied to $u_2$ gives
\[
\essinf_{\mathcal Q} u_2 \ge \delta_2\omega,
\]
which implies \eqref {S2:Lmainbothoss} in this case.
The proof is completed by observing that $Q_{\sigma\omega, \lambda\rho} \subset\mathcal Q$.
\end {proof}

For our H\"older continuity estimates, we define a time scale in terms of the function $G$, the function $u$ and the set $\Omega$ on which $u$ is defined.  We shall now include $u$ and $\Omega$ in the notation for simplicity.  Specifically, for any real number $\tau$, we define
\begin {subequations} \label {S2:timescale}
\begin {gather}
|\tau|_G= \frac {U}{G^{-1}(U^2/|\tau|)}, \\
\intertext {where}
U= \essosc_{\Omega} u.
\end {gather}
\end {subequations}
With this time scale, we define the parabolic distance between two sets such $\mathcal{K}_{1}$ and $\mathcal {K}_2$ by
\[
\dist_P (\mathcal{K}_{1};\mathcal{K}_{2}) := \inf_{\substack {(x,t) \in \mathcal{K}_{1}\\
 (y,s)\in \mathcal{K}_{2}\\
 s\le t} }
 \left(|x-y| + |t-s|_{G} \right).
\]
(Note that, strictly speaking, this quantity is not a distance because it is not symmetric with respect to the order in which we write the sets.
Nonetheless, the terminology of distance is useful as a suggestion of the technically correct situation.)

Because of the generalized function $G$, it is natural to obtain a modulus of continuity in terms of $G$. We are also able to derive a H\"{o}lder estimate written in terms of exact powers.

\begin{theorem}\label{S2:T:NaturalConty}
Let $u$ be a bounded weak solution of \eqref{I:gen} with \eqref{I:gen-str} in $\Omega$, and suppose $g_1\le2$ or $g_0\ge2$. Then $u$ is locally continuous. Moreover, there exist constants $\gamma$ and $\alpha\in(0,1)$ depending only upon the data such that, for any two distinct points $(x_1,t_1)$ and $(x_2,t_2)$ in any subset $\Omega'$ of $\Omega$ with $\dist_P(\Omega'; \partial_{p}\Omega)$ positive, we have
\begin {equation} \label {S2:T:NaturalConty:E1}
\left|u(x_1,t_1) - u(x_2,t_2)\right| \leq \gamma U \left(\frac{|x_1 - x_2|+ |t_1 - t_2|_{G}}{\dist_P(\Omega'; \partial_{P}\Omega)}\right)^{\alpha}.
\end {equation}
In addition (with the same constants),
\begin {equation} \label {S2:T:NaturalConty:E2}
\left|u(x_1,t_1) - u(x_2,t_2)\right| \leq \gamma U \left(\frac{|x_1 - x_2|+ (U/G^{-1}(U^2))\max\{|t_1 - t_2|^{1/g_0}, |t_1-t_2|^{1/g_1}\}}{\dist_P(\Omega'; \partial_{P}\Omega)}\right)^{\alpha}.
\end {equation}
\end{theorem}

\begin{proof}
If $U=0$, then this result is true for any choice of $\gamma$ and $\alpha$, so we assume that $U>0$ and set $\omega_0=U$.  We also set
\begin{equation}\label{E:r}
\rho_0 = \frac 12 \inf_{\substack{(x,t)\in \partial_{p}\Omega\\ s\le t_1}}\left(|x-x_1|+|t-t_1|_G\right).
\end{equation}
We then define $\varepsilon = \min\{\lambda, \frac 12\sigma^{(2-g_0)/g_0}\}$ (where $\lambda$ and $\sigma$ are the constants from Lemma~\ref {S2:Lmainboth}, 
\[
\rho_n=\varepsilon^n\rho_0, \quad \omega_n= \sigma^n\omega_0,
\]
and define a sequence of cylinders $(Q_n)$ by 
\[
Q_n=Q_{\sigma_n,\rho_n}^{x_1,t_1}.
\]
It's easy to check that $Q_0\subset \Omega$ and that $Q_{n+1}\subset Q_n$ for any $n$.
Combining Lemma \ref {S2:Lmainboth} with an easy induction, we find that
\[
\essosc_{Q_n} u\le \omega_n
\]
for any $n$.
If $(x_2,t_2) \in Q_0$ with $x_1\neq x_2$ and $t_1\neq t_2$, then there are nonnegative integer $n$ and $m$ such that
\begin{subequations}\label{S2:rs}
\begin {gather}
\rho_{n+1} < |x_1-x_2|  \leq \rho_{n}, \label {S2:r}\\
\omega_{m+1}^{2} G\left(\frac{\omega_{m+1}}{\rho_{m+1}}\right)^{-1} < |t_1-t_2| \leq  \omega_{m}^{2} G\left(\frac{\omega_{m}}{\rho_{m}}\right)^{-1}. \label {S2:s}
\end {gather}
\end {subequations}
As a result, we obtain that
\[
 |u(x_1,t_1)-u(x_2,t_2)| \leq \max\{ \omega_{n}, \omega_{m}\}.
\]

From the first inequality of (\ref{S2:r}), we derive
\[
\frac{|x_1-x_2|}{\rho_0} > \varepsilon^{n+1} = \left(\sigma^{\log_{\sigma} \varepsilon }\right)^{n+1}
\]
which implies
\[
\omega_{n} = \sigma^{n}\omega_{0} < \sigma^{-1}\omega_{0} \left(\frac{|x_1-x_2|}{\rho_{0}}\right)^{\alpha_1}
\]
for $\alpha_1 = \log_{\varepsilon} \sigma$.

On the other hand, the first inequality of (\ref{S2:s}) implies that
\[
|t_1-t_2|_G \ge \frac {U}{G^{-1}\left( U^2\omega^{-2}_{m+1} G\left(\frac {\omega_{m+1}}{\rho_{m+1}}\right)\right)}.
\]
We now estimate the expression in the denominator of this fraction:
\begin {align*} 
 U^2\omega^{-2}_{m+1} G\left(\frac {\omega_{m+1}}{\rho_{m+1}}\right) &=
\sigma^{-2(m+1)} G\left(\frac {\omega_{m+1}}{\rho_{m+1}}\right) \\
&\ge G\left(\sigma^{-2(m+1)/g_0}\left(\frac \sigma\varepsilon\right)^{m+1}
\frac {\omega_0}{\rho_0}\right).
\end {align*}
To proceed, we now define
\[
\beta=\varepsilon \sigma^{(g_0-2)/g_0}
\]
and note that the choice of $\varepsilon$ implies that $\beta<1$. 
We then have that
\[
|t_1-t_2|_G \ge \frac {U}{\beta^{-(m+1)} \frac {\omega_0}{\rho_0}} =\beta^{m+1} \rho_0.
\]
Hence by letting $\alpha_2 = \log_{\beta} \sigma $, we have
\[
\omega_{m} \leq \beta^{\alpha_2m}\omega_0 \le
 \left(\frac {|t_1-t_2|_G}{\beta\rho_0}\right)^{\alpha_2} \omega_0.
\]
Therefore, for some $\gamma >0$,
\[
 |u(x_1,t_1)-u(x_2,t_2)| \leq \gamma U\left[ \left(\frac{|x_1-x_2|}{\rho_0}\right)^{\alpha_1} + \left(\frac{|t_1-t_2|_G}{\rho_0}\right)^{\alpha_2} \right].
\]
This inequality implies \eqref {S2:T:NaturalConty:E1} with $\alpha=\min\{\alpha_1,\alpha_2\}$ because $\rho_0 \ge \dist_P(\Omega';\partial_P\Omega)$.
If $x_1=x_2$ or if $t_1=t_2$, then a similar (but simpler) argument yields
the result.

If $(x_2,t_2)\notin Q_0$, then $|x_1-x_2|+|t_1-t_2|_G \ge \frac12\rho_0$, so \eqref {S2:T:NaturalConty:E1} follows, using the choice of $\alpha$ above, by taking $\gamma\ge2^{1+\alpha}$.

To prove \eqref {S2:T:NaturalConty:E2}, we consider two cases.  First, if $|t_1-t_2| \le 1$, then
\[
G\left( \frac {G^{-1}(U^2)}{|t_1-t_2|^{1/g_0}}\right) \ge \frac 1{|t_1-t_2|} U^2,
\]
so
\[
|t_1-t_2|_G \le \frac {U}{G^{-1}(U^2)}|t_1-t_2|^{1/g_0}.
\]
Second, if $|t_1-t_2|>1$, then
\[
G\left( \frac {G^{-1}(U^2)}{|t_1-t_2|^{1/g_1}}\right) \ge \frac 1{|t_1-t_2|} U^2,
\]
so
\[
|t_1-t_2|_G \le \frac {U}{G^{-1}(U^2)}|t_1-t_2|^{1/g_1}.
\]
Combining these inequalities with \eqref {S2:T:NaturalConty:E1} then gives \eqref {S2:T:NaturalConty:E2}.
\end{proof}

\section{Proof of the Main Lemma and the two alternatives}\label{S3}
Throughout this section, let $u$ to be a bounded nonnegative weak solution of \eqref{I:gen} with \eqref{I:gen-str}. The proof of Lemma~\ref{S2:MainLemma} is composed of four steps under the assumption that $u$ is large at least half of a cylinder $Q_{\omega, 2R}$. Then Proposition~\ref{S3:prop1} implies that a spatial cube at some fixed time level is found on which $u$ is away from its minimum (zero value) on arbitrary fraction of the spatial cube. From the spatial cube, positive information spread in both later time and over the space variables with time limitations (Proposition~\ref{S3:prop2} \& Proposition~\ref{S3:prop3}). Controlling the positive quantity $\theta >0$ on $T_{k,\rho}$ is key to overcoming those time restrictions. Once we have a subcylinder centered at $(0, 0)$ in $Q_{\omega, 4R}$ with arbitrary fraction of the subcylinder, we finally apply modified De Giorgi iteration (Proposition~\ref{S3:prop4}) to obtain strictly positive infimum of $u$ in a smaller cylinder around $(0,0)$. We can carry analogous proof when $u$ is away from its maximum ($u$ is close to its minumum) at least half of cylinder.

\subsection {The basic results}
Our first proposition shows that if a nonnegative function is large on part of a cylinder, then it is large on part of a fixed cylinder. Except for some minor variation in notation, our result is Lemma 7.1 from Chapter III of \cite {DB93}; we include a proof for completeness.

\begin {proposition} \label {S3:prop1} Let $k$, $\rho$, and $T$ be positive constants.  If $u$ is a measurable nonnegative function defined on $Q=K_\rho\times(-T,0)$ and if there is a constant $\nu_1\in [0,1)$ such that
\[
|Q\cap \{u<k\}| \le (1-\nu_1)|Q|, 
\]
then there is a number
\[
\tau_1\in \left(- T, -\, \frac {\nu_1}{2-\nu_1}T\right)
\]
for which
\[
\left| \{ x\in K_\rho: u(x,\tau_1)<k\}\right| \le \left(1-\frac {\nu_1}2\right) |K_\rho|.
\]
\end {proposition}
\begin {proof}
To simplify the notation, we set $\tau = \frac {\nu_1}{2-\nu_1}T$.

If there were no such $\tau_1$, then we would have
\begin {align*}
|Q\cap \{u>k\}| &= \int_{-T}^0|\{x\in K_\rho: u(x,t)>k\}|\, dt \\
&\ge  \int_{-T}^{-\tau}|\{x\in K_\rho: u(x,t)>k\}|\, dt \\
&>\left( 1-\frac {\nu_1}2\right)\left(1-\frac {\nu_1}{2-\nu_1}\right)|K_\rho| T \\
&=(1-\nu_1)|Q|.
\end {align*}
\end {proof}

Our next proposition is similar to Lemmata III.4.1, III.7.2, IV.10.2 from \cite{DB93}. If $g_0 >2$, then the next proposition can be replaced by Corollary 3.4 from \cite{GiSuVe10} which does not involve the logarithmic energy estimate.

\begin{proposition}\label{S3:prop2}
Let  $\nu$, $k$, $\rho $, and $\theta$ be given positive constants with $\nu<1$. Then, for any $\epsilon \in (0,1)$, there exists a constant $\delta =\delta (\nu,\epsilon, \theta, \data)$ such that, if $u$ is a nonnegative supersolution of \eqref {I:gen} in $K_\rho\times(-\tau,0)$ with
\begin{equation}\label{S3:prop2-hyp}
\left|\left\{x\in K_{\rho}: u(x,-\tau) < k \right\}\right|< (1-\nu) \left|K_{\rho}\right|
\end{equation}
for some
\begin {subequations} \label{S3:prop2-tau}
\begin{align}
\tau &\leq  \theta(\delta k)^2 G\left(\frac{\delta k}{\rho}\right)^{-1} \quad \text{if} \quad g_1 \leq 2, \label{S3:prop2-tauge} \\
\tau &\leq  \theta k^2 G\left(\frac{k}{\rho}\right)^{-1} \quad \text{if} \quad g_0 \geq 2,\label{S3:prop2-taule}
\end{align}
\end {subequations}
then
\[
\left|\left\{ x\in K_{\rho}: u(x, -t) < \delta k \right\}\right| < \left(1- (1-\epsilon)\nu\right)|K_{\rho}|
\]
for any $-t \in (-\tau, 0]$.
\end{proposition}
\begin{proof}
Here we apply the logarithmic energy estimate \eqref{S4:LogE} in a parabolic cylinder $K_{\rho} \times [-\tau, -s]$ for any $-s \in [0,-\tau)$. For some $\sigma \in (0,1)$ to be determined later, we introduce a piecewise linear cutoff function independent of the time variable, that is, 
\[
\zeta = \begin{cases}
        1 & \text{inside} \ K_{(1-\sigma) \rho} \times [-\tau, -s] \\
        0 & \text{on the lateral boundary of }\ K_{\rho} \times [-\tau, -s].
        \end{cases}
\]
satisfying
\[
|D\zeta| \leq \frac{1}{\sigma \rho}, \quad \zeta_{t} = 0 .
\]
From \eqref{S4:LogE} by letting $q = g_1$, it follows that 
\begin{equation}\label{LogE1}\begin{split}
& \int_{K_{\rho}\times \{-s\}} H(\Psi^2) \zeta^{g_1} \,dx \leq \int_{K_{\rho}\times \{-\tau\}} H(\Psi^2) \zeta^{g_1} \,dx \\
&\quad + 2 C_{0}^{1-g_1} C_{1}^{g_1} g_{1}^{2g_1} \int_{-\tau}^{-s}\int_{K_{\rho}} h(\Psi^2)|\Psi||\Psi'|^2 G\left(\frac{|D\zeta|}{\Psi'}\right) \,dx\,dt,
\end{split}\end{equation}
where $h$ and $H$ are defined in Lemma~\ref{S1:H-ineq}.
Let $\delta=2^{-j}$ where $j$ is to be chosen large enough. We recall
\[
\Psi = \ln^{+}\left[\frac{k}{(1+\delta) k - (u-k)_{-}}\right], \ \Psi' = \frac{1}{(u-k)_{-} - (1+\delta) k}
\]
that becomes zero when $u \geq (1-\delta)k$. 
Since $0 \leq (u-k)_{-} \leq k$, we also have
\[
\Psi \leq \ln^{+} \delta^{-1} = j \ln 2, \quad \frac{1}{(1+\delta) k}\leq |\Psi'| \leq \frac{1}{\delta k}.
\]

The first integral term on the right hand side of \eqref{LogE1} is bounded by
\[\begin{split}
&\int_{K_{\rho}\times \{-\tau\}} H(\Psi^2) \zeta^{g_1} \,dx \\
&\leq H\left(j^2 (\ln 2)^2\right)\left| \{x \in K_{\rho}: u(x,-\tau) \leq (1-\delta)k \} \right| \\
&\leq (1-\nu) H\left(j^2 (\ln 2)^2\right) \left| K_{\rho} \right|
\end{split}\]
because of the assumption \eqref{S3:prop2-hyp}.

Now to handle the second integral on the right hand side of \eqref{LogE1}, we make observations of upper bounds of the quantity 
\[
|\Psi'|^2 G\left(\frac{|D\zeta|}{\Psi'}\right)
\]
depending on the range of constants $g_0$ and $g_1$. 
When $g_1 \leq 2$, first note that $\delta k |\Psi'| \leq 1$ and therefore
\[\begin{split}
|\Psi'|^2 G\left(\frac{|D\zeta|}{\Psi'}\right) 
&\leq \left( \delta k |\Psi'|\right)^{2-g_1} \left( \delta k \right)^{-2} G\left( \delta k |D\zeta| \right) \\
&\leq \sigma^{-g_1} \left(\delta k\right)^{-2} G\left( \frac{\delta k}{\rho} \right)
\end{split}\]
because $2-g_1 \geq 0$ and $\sigma < 1$. In this case it is natural to choose $\tau$ to satisfy \eqref {S3:prop2-tauge}.

If $g_0 \geq 2$, then we use the inequalities $1 \leq (1+\delta)k |\Psi'|$ and $\delta < 1$. Hence we derive that 
\begin{align*}
|\Psi'|^2 G\left(\frac{|D\zeta|}{\Psi'}\right) 
&\leq \left( (1+\delta) k |\Psi'|\right)^{2-g_0} \left( (1+\delta) k \right)^{-2} G\left( (1+\delta) k |D\zeta| \right) \\
&\leq   2^{g_1} \sigma^{-g_1} k ^{-2} G\left( \frac{k}{\rho} \right)
\end{align*}
because $2-g_0 \leq 0$ and $1 < 1+\delta <2 $. Therefore \eqref {S3:prop2-taule} yields that, for any $-s \in (-\tau, 0]$ (which implies that $|\tau- s| \leq \tau$), we have
\begin{align*}
& \int_{-\tau}^{-s}\int_{K_{\rho}} h(\Psi^2)|\Psi||\Psi'|^{2}G\left(\frac{|D\zeta|}{|\Psi'|}\right) \,dx\,dt \\
&\leq  2^{g_1}\theta h\left(j^2 (\ln 2)^2 \right) ( j \ln 2 ) \sigma^{-g_1}  |K_{\rho}| \\
&\leq 2^{g_1}\theta \frac{H(j^2 (\ln 2)^2)}{j \ln 2} \sigma^{-g_1} |K_{\rho}|,
\end{align*}
because $2^{g_1}\ge1$.

To obtain the lower bound of the left hand side of \eqref{LogE1}, we integrate over the smaller set $\{u< \delta k\}$ which gives $\Psi \geq \ln^{+} \left[ \frac{1}{2\delta }\right]$, that is, in the set $\{u< 2^{-j} k\}$
\[
\Psi \geq \ln^{+} (2\delta)^{-1} = (j-1) \ln 2.
\]
Therefore the left hand side of the inequality \eqref{LogE1} is lower bounded by
\begin{equation*}\begin{split}
& \int_{K_{\rho}\times \{-s\}} H(\Psi^2) \zeta^{g_1} \,dx \\
& \geq H\left((j-1)^2 (\ln 2)^2\right) \left|\left\{x\in K_{(1-\sigma)\rho}:u(x,-s)<\delta k \right\}\right|.
\end{split}\end{equation*}

It follows that
\begin{align*}
&\left|\left\{x\in K_{\rho}: u(x,-s)<\delta k \right\}\right| \\
&\leq \left| \left\{ x\in K_{(1-\sigma)\rho}: u(x,t) < \delta k \right\}\right| + \left| K_{\rho} \setminus K_{(1-\sigma)\rho}\right| \\
&\leq \left[(1-\nu)\frac{H\left(j^2 (\ln 2)^2\right)}{H\left((j-1)^2(\ln 2)^2\right)} + \frac{C\theta H\left(j^2 (\ln 2)^2\right)}{j \sigma^{g_1} H\left((j-1)^2(\ln 2)^2\right)} + N\sigma\right] |K_{\rho}|
\end{align*}
upon combining upper bounds of \eqref{LogE1}, where $C$ depends on $C_0$, $C_1$, and $g_1$. 
For brevity, set
\[
H_0 = \frac{H\left(j^2 (\ln 2)^2\right)}{H\left((j-1)^2(\ln 2)^2\right)}.
\]
For any given $\epsilon \in (0,1)$, we choose an integer $j$ large enough and $\sigma\in(0,1)$ small enough so that the following three inequalities hold:
\begin{subequations}\label{S3:P2:jsig}
\begin{gather}
H_0 \leq 1 + \epsilon\nu, \label{S3:P2:j01} \\
\frac{C \theta H_0}{j \sigma^{g_1} } \leq \frac{\epsilon\nu^2}{2}, \label{S3:P2:j02} \\
N\sigma \leq \frac{\epsilon\nu^2}{2} .\label{S3:P2:sig}
\end{gather}
\end{subequations}
Then inequalities \eqref{S3:P2:jsig} yield our conclusion.

Now we complete the proof by going back to \eqref{S3:P2:jsig} and finding $j$ and $\sigma$.
From \eqref{S3:P2:sig}, first fix
\[
\sigma = \frac{\epsilon\nu^2}{2N}.
\]
Then assuming \eqref{S3:P2:j01}, the inequality \eqref{S3:P2:j02} holds if
\[
j  \geq \frac{C \theta (1+\epsilon \nu)}{2\sigma^{g_1} \epsilon \nu^2},
\]
which gives
\[
j \geq \frac{C (1+\epsilon\nu) (4N)^{g_1} }{2 \epsilon^{1+g_1}\nu^{2(1+g_1)}}.
\]
It is sufficient to choose
\[
j \geq C(C_0, C_1, g_1, N) \left(\epsilon \nu^2\right)^{-1-g_1}.
\]
Finally, \eqref{S3:P2:j01} is satisfied if $j$ is so large that
\[
\left(\frac{j}{j-1}\right)^{g_1} \leq 1 + \epsilon\nu ,
\]
which is equivalent to
\[
j> \frac {(1+\epsilon\nu)^{1/g_1}}{(1+\epsilon\nu)^{1/g_1}-1}.
\]
So the proof is completed by taking $j$ to be any integer greater than
\[
\max \left\{ C(C_0, C_1, g_1, N)\theta\left(\epsilon \nu^2\right)^{-1-g_1}, \frac {(1+\epsilon\nu)^{1/g_1}}{(1+\epsilon\nu)^{1/g_1}-1} \right\}.
\]
\end{proof}


The following proposition is spreading positivity over the space. When we have some portion of positive data along all the time, then a mixture of Poincar{\'e}'s inequality and a local energy estimate generates arbitrary fractional control over a cylinder with somewhat wider side length.
Especially when $g_1 >2$, somewhat large length of the time interval for the initially given positive data collected place is required to carry spreading positivity properly. Proposition~\ref{S3:prop3} is analogous to Lemma 3.5 from \cite{GiSuVe10}, Theorem 1.1 from \cite{DBGiVe10}, Proposition 6.1 from \cite{DBGiVe08}, and Lemma IV.11.1 from \cite{DB93}.

\begin{proposition}\label{S3:prop3}
Let $k$ and $\rho$ be positive numbers and suppose $u$ is a nonnegative supersolution of \eqref {I:gen} in $K_{2\rho}\times(-2\tau,0)$ for some $\tau$ satisfying
\begin{subequations}\label{S3:P3:tau}
\begin{align}
\tau &\geq \theta k^2 G\left(\frac{k}{\rho}\right)^{-1} \quad \text{if} \quad  g_1 \leq 2, \\
\tau &\geq \theta \left(\delta^{*} k\right)^2 G\left(\frac{\delta^{*}k}{\rho}\right)^{-1}\quad   \text{if} \quad  g_0 \geq 2.
\end{align}
\end{subequations}
Then for any $\nu$ and $\alpha$ in $(0,1)$ and any $\theta>0$, there exists a constant $\delta^{*}=\delta^{*}(\alpha,\nu, \min\{1,\theta\},\data)\in (0,1)$ such that, if
\begin{equation}\label{S3:prop3-hyp}
\left|\left\{x\in K_{\rho}: u(x, t) < k \right\}\right| < (1-\alpha) |K_{\rho}|
\end{equation}
for all $t \in (-2\tau, 0]$, then we have 
\begin{equation}\label{S3:P3:con}
\left|\left\{(x,t)\in K_{\rho}\times [-\tau, 0]: u(x,t) < \delta^{*} k \right\} \right|< \nu \left|K_{\rho}\times [-\tau,0]\right|.
\end{equation}
\end{proposition}

\begin{proof}
Let $k_{j}=2^{-j}k$ for $j=0,1,2,\ldots,j^{*}$ with $j^{*}$ to be determined later. Denote $\delta^{*} = 2^{-j^{*}}$. For simplicity, denote
\[
 A_{j} = \left\{ (x,t) \in K_{\rho}\times [-\tau, 0] : u(x,t) < k_{j} \right\}.
\]

We work with a piecewise linear cutoff function
\[
\zeta = \begin{cases}
        1 & \text{inside of } \ K_{\rho}\times [-\tau, 0] \\
        0 & \text{on the parabolic boundary of } \ K_{2\rho}\times [-2\tau, 0]
        \end{cases}
\]
with
\[
|D\zeta| \leq \frac{1}{\rho}, \quad \zeta_{t} \leq \frac{1}{\tau}.
\]
The local energy estimate \eqref{S4:LocalE} (by ignoring the first term on the left hand side) provides
\begin{equation}\label{S3:P3:local}\begin{split}
& \int_{-2\tau}^{0}\int_{K_{2\rho}} G\left(|D(u- k_{j} )_{-}|\right) G^{r-1}\left(\frac{\zeta (u-k_{j})_{-}}{\rho}\right) (u-k_{j})_{-}^{s} \zeta^{q} \,dx\,dt \\
&\leq \gamma_1 \int_{-2\tau}^{0}\int_{K_{2\rho}} G^{r-1}\left(\frac{\zeta (u-k_{j})_{-}}{\rho}\right) (u-k_{j})_{-}^{s+2} \zeta^{q-1} \zeta_{t} \,dx\,dt \\
&\quad + \gamma_2 \int_{-2\tau}^{0}\int_{K_{2\rho}} G^{r}\left(\frac{\zeta (u-k_{j})_{-}}{\rho}\right) (u-k_{j})_{-}^{s} \zeta^{q-1-2g_1} \,dx\,dt.
\end{split}\end{equation}

Here note that for all $j=0,\ldots,j^{*}$
\[
k_{j}^{2} \zeta_{t} \leq \frac 1\theta G\left(\frac{k_{j}}{\rho}\right)
\]
because \eqref{S3:P3:tau} implies that, for any $j=0, \ldots, j^{*}$,
\[
\tau \geq k_{j}^{2} G\left(\frac{k_{j}}{\rho}\right)^{-1}.
\]
The integral estimate \eqref{S3:P3:local} simplifies to
\begin{equation}\label{S3:P3:local02}
\int_{-\tau}^{0}\int_{K_{\rho}} G\left(|D(u-k_{j})_{-}|\right) \,dx\,dt
\leq \gamma G\left(\frac{k_{j}}{\rho}\right) \left| K_{2\rho} \times [-2\tau, 0] \right|.
\end{equation}

Owing to the assumption (\ref{S3:prop3-hyp}), we may apply the Poincar{\'e} type inequality, Corollary~\ref{S4:Poincare-cor}. For any $t \in [-\tau, 0]$, it follows that
\[\begin{split}
&\left(k_{j} - k_{j+1}\right) \left|\left\{x\in K_{\rho}: u(x,t) < k_{j+1} \right\}\right| \\
&\quad\leq \frac{\gamma\rho^{N+1}}{\theta \alpha\rho^{N}} \int_{K_{\rho}\cap \{k_{j+1}\leq u < k_{j}\}} |D(u-k_{j})_{-}| \,dx.
\end{split}\]
Note $k_{j} - k_{j+1} = k_{j+1}$. After integrating over the time variable from $-\tau$ to $0$, we obtain
\begin{equation}\label{S3:P3:p01}
\frac{k_{j+1}}{\rho} \left|A_{j+1}\right| \leq \frac{\gamma}{\theta\alpha \eta^{N}} \iint_{A_{j} \setminus A_{j+1}} |D(u-k_{j})_{-}|  \,dx \,dt.
\end{equation}
After dividing \eqref{S3:P3:p01} by $|A_{j} \setminus A_{j+1}|$ and assuming (without loss of generality) that the constant $\gamma$ in this inequality is at least $1$, we apply Jensen's inequality and Lemma ~\ref {S1:ineq}(b) to infer that
\begin{equation}\label{S3:P3:p02}
G\left(\frac{|A_{j+1}|}{|A_{j}\setminus A_{j+1}|} \frac{k_{j+1}}{\rho}\right) \leq \frac{\gamma^*}{ |A_{j}\setminus A_{j+1}|} \iint_{A_{j} \setminus A_{j+1}} G\left(|D(u-k_{j})_{-}|\right)  \,dx \,dt
\end{equation}
with
\[
\gamma^*= \left(\frac \gamma{\min\{1,\theta\}\alpha }\right)^{g_1}.
\]

Because of \eqref{S3:P3:local02}, the inequality \eqref{S3:P3:p02} generates
\begin{equation}\label{S3:P3:p03}
G\left(\frac{|A_{j+1}|}{|A_{j}\setminus A_{j+1}|} \frac{k_{j+1}}{\rho}\right) \leq \gamma 2^{N+1}\gamma^* \frac{\left|K_{\rho} \times [-\tau, 0]\right|}{|A_{j}\setminus A_{j+1}|} G\left(\frac{k_{j}}{\rho}\right).
\end{equation}
Denote $\Omega_{\tau} := K_{\rho} \times [-\tau, 0]$. There are two cases to consider for any $j$: either
\begin{gather*}
|A_{j+1}| > |A_{j}\setminus A_{j+1}|, \\
\intertext{or}
|A_{j+1}| \le |A_{j}\setminus A_{j+1}|.
\end{gather*}

First, if $|A_{j+1}| > |A_{j}\setminus A_{j+1}|$, then we have
\[
\left(\frac{|A_{j+1}|}{|A_{j}\setminus A_{j+1}|}\right)^{g_0} 2^{-g_1} G\left(\frac{k_{j}}{\rho}\right) \leq G\left(\frac{|A_{j+1}|}{|A_{j}\setminus A_{j+1}|} \frac{k_{j+1}}{\rho}\right).
\]
Therefore, \eqref{S3:P3:p03} generates
\begin{equation}\label{S3:P3:c022}
\left(\frac{|A_{j+1}|}{|\Omega_{\tau}|}\right)^{\frac{g_0}{g_0 -1}} \leq \gamma(\gamma^* )^{\frac{1}{g_0 -1}} \frac{|A_{j}\setminus A_{j+1}|}{|\Omega_{\tau}|}.
\end{equation}

Second, if $|A_{j+1}| \leq |A_{j}\setminus A_{j+1}|$, then we observe that
\[
\left(\frac{|A_{j+1}|}{|A_{j}\setminus A_{j+1}|}\right)^{g_1} 2^{-g_1} G\left(\frac{k_{j}}{\rho}\right) \leq G\left(\frac{|A_{j+1}|}{|A_{j}\setminus A_{j+1}|} \frac{k_{j+1}}{\rho}\right).
\]
The inequality \eqref{S3:P3:p03} gives
\[ 
\left(\frac{|A_{j+1}|}{|A_{j}\setminus A_{j+1}|}\right)^{g_1} 2^{-g_1}  \leq \gamma \gamma^* \frac{\left|\Omega_{\tau}\right|}{|A_{j}\setminus A_{j+1}|},
\] 
and hence
\[ 
\left(\frac{|A_{j+1}|}{|\Omega_{\tau}|}\right)^{\frac{g_1}{g_1 -1}} \leq \gamma (\gamma^*)^{\frac{1}{1-g_1}} \frac{|A_{j}\setminus A_{j+1}|}{|\Omega_{\tau}|}.
\] 
Since $|A_{j+1}|/\Omega_\tau|\le1$ and $g_1/(g_1-1)\le g_0/(g_0-1)$, it follows that
\[
\left(\frac{|A_{j+1}|}{|\Omega_{\tau}|}\right)^{\frac{g_0}{g_0 -1}} \le\left(\frac{|A_{j+1}|}{|\Omega_{\tau}|}\right)^{\frac{g_1}{g_1 -1}}.
\]
In addition, since $\gamma^*\ge1$ and $1/(g_1-1)\le 1/(g_0-1)$, it follows that
\[
(\gamma^*)^{1/(g_1-1)} \le (\gamma^*)^{1/(g_0-1)}.
\]
Therefore, \eqref {S3:P3:c022} is valid for all $j\in\{0,\dots,j^*-1\}$.

Next we take the sum for $j=0, \ldots, j^{*} - 1$ of the inequality \eqref{S3:P3:c022}. Noting that $|A_{j^{*}}| \leq |A_{j+1}|$ for all $j=0,\ldots, j^{*}-1$, we conclude that
\[
j^*\left( \frac {|A_{j^*}|}{|\Omega_\tau|}\right)^{g_0/(g_0-1)} \le \gamma(\gamma^*)^{1/(g_0-1)}.
\]

We now reach  our conclusion \eqref{S3:P3:con} by choosing $j^{*}$ such that
\[
j^{*} \geq \frac 1\gamma\nu ^{g_0/(1-g_0)} (\gamma^*)^{1(1-g_0)}. 
\]
\end{proof}

The following proposition is modified DeGiorgi iteration with generalized structure conditions \eqref{I:gen-str}. Basically, our Proposition~\ref{S3:prop4} is equivalent to Lemmata III.4.1, III.9.1, IV.4.1 from \cite{DB93}.

\begin{proposition}\label{S3:prop4}
For a given positive constant $\theta$,  there exists $\nu_0 = \nu_{0} (\theta, \data) \in (0,1)$ such that, if $u$ is a nonnegative supersolution of \eqref {I:gen} in $Q_{k,2\rho}(\theta)$ with
\begin {equation} \label {S3:prop4:hyp}
\left|\left\{(x,t)\in Q_{k,2\rho}(\theta): u(x,t) < k \right\}\right| < \nu_0 \left|Q_{k,2\rho}(\theta)\right|
\end {equation}
for some positive constants $k$ and $\rho$, then
\[
\essinf_{Q_{k,\rho}(\theta)} u(x,t) \geq \frac{k}{2}.
\]
\end{proposition}

\begin{proof}
First, we construct two sequences $\{\rho_{n}\}_{n=0}^{\infty}$ and $\{k_{n}\}_{n=0}^{\infty}$ such that
\begin{equation*}
\rho_n = \rho + \frac{\rho}{2^{n}} \ \text{and} \ k_n = \frac{k}{2} + \frac{k}{2^{n+1}} \ \text{for}\ n=0,1,\ldots.
\end{equation*}
Because $G(\sigma)$ is an increasing function, the sequence $\{Q_{n}\}_{n=0}^{\infty}$, given by
\begin{equation*}
Q_n = K_{\rho_n} \times [-T_{k,\rho_{n}}(\theta), 0],
\end{equation*}
is a nested and shrinking sequence of cylinders.
Let us take a sequence of piecewise linear cutoff functions $\{\zeta_{n}\}_{n=0}^{\infty}$ such that
\begin{equation*}
\zeta_n =
\begin{cases}
1 & \text{inside of } \ Q_{n+1} \\
0 & \text{on the parabolic boundary of} \ Q_{n},
\end{cases}
\end{equation*}
satisfying
\begin{gather*}
|D\zeta_n| \leq \frac{2^{n+1}}{\rho}, \\
0 \leq (\zeta_n)_t \leq \frac 1{\theta k^2(G(\frac k{\rho_n})^{-1}- G(\frac k{\rho_{n+1}})^{-1})}
\end{gather*}
Let us note that
\[
\frac {2^{n+1}}{\rho}\ge\frac {2^{n+2}}{\rho_n}
\]
because $\rho_n\le 2\rho$.
We also need a different upper bound for $(\zeta_n)_t$.
As a first step, we write
\[
G\left(\frac k{\rho_n}\right)^{-1}- G\left(\frac k{\rho_{n+1}}\right)^{-1}= \int_{\rho_{n+1}}^{\rho_n} \frac k{s^2}g\left(\frac ks\right) G\left(\frac ks\right)^{-2}\, ds.
\]
By using the first inequality in \eqref {I:DeltaNabla} and Lemma~\ref {S1:ineq}(b), we conclude that
\[
\frac k{s^2}g\left(\frac ks\right) G\left(\frac ks\right)^{-2} \ge \frac {g_0}sG\left(\frac ks\right)^{-1} \\
\ge \frac {g_0}{\rho_n} G\left(\frac k{\rho_n}\right)^{-1}
\]
for any $s\in(\rho_{n+1},\rho_n)$ and hence
\begin {align*}
(\zeta_n)_t &\le \frac 1{\theta k^2g_0} G\left(\frac k{\rho_n}\right)\frac {\rho_n}{\rho_n-\rho_{n+1}} \\
&= \frac {2^n+2}{g_0\theta}k^{-2}G\left(\frac k{\rho_n}\right) \\
&\le \frac {2^{n+1}}{g_0\theta}k^{-2}G\left(\frac k{\rho_n}\right).
\end {align*}
Note that
\[
G(|D\zeta_{n}|\zeta_{n}(u-k_{n})_{-}) \leq 2^{(n+1)g_1} G\left(\frac{\zeta_{n}(u-k_{n})_{-}}{\rho_{n}}\right).
\]
Therefore, the local energy estimate \eqref{S4:LocalE} yields, for some constants $\gamma_0$ and $\gamma_1$, that
\begin{equation}\label{S3:prop4-E01}\begin{split}
    & \sup_{t}\int_{K_{\rho_{n}}} G^{r-1} \left(\frac{\zeta_{n} (u-k_{n})_{-}}{\rho_{n}}\right) (u-k_{n})_{-}^{s+2} \zeta_{n}^{q} \,dx \\
    &\quad + \iint_{Q_{n}} G\left(|D(u-k_{n})_{-}|\right) G^{r-1} \left(\frac{\zeta_{n} (u-k_{n})_{-}}{\rho_{n}}\right) (u-k_{n})_{-}^{s} \zeta_{n}^{q} \,dx\,dt \\
    &\leq  \gamma_{0}\iint_{Q_{n}} G^{r-1} \left(\frac{\zeta_{n} (u-k_{n})_{-}}{\rho_{n}}\right) (u-k_{n})_{-}^{s+2} \zeta_{n}^{q-1} (\zeta_{n})_{t} \,dx\,dt \\
    &\quad + \gamma_{1} 2^{(n+1)g_1} \iint_{Q_{n}} G^{r} \left(\frac{\zeta_{n} (u-k_{n})_{-}}{\rho_{n}}\right) (u-k_{n})_{-}^{s}  \,dx\,dt.
\end{split}\end{equation}
We now observe that
\begin{equation*}
(u-k_n)_{-} = \max \{0, k_n - u \} \leq k_n \leq k ,
\end{equation*}
and  that $G^{r-1}(\sigma)\sigma^{s+2}$ and $G^{r}(\sigma)\sigma^{s}$ are increasing with respect to $\sigma$.
Since $q\ge 1$, we conclude that the right hand side of \eqref{S3:prop4-E01} is bounded by
\[
RHS \leq \left\{ \gamma_0 \frac{2^{n+1}}{g_0 \theta} + \gamma_1 2^{(n+2)g_1} \right\}G^{r}\left(\frac{k}{\rho_{n}}\right)k^{s} \left| A_n\right|,
\]
where $A_n= Q_n\cap \{u>k_n\}$.

Using $u_{n}:= (u-k_{n})_{-}$ for simpler notation, we obtain that
\begin{equation}\label{S3:prop4-E02}\begin{split}
& 2^{-2(n+2)}k^2 G\left(\frac{k}{\rho_{n}}\right)^{-1} \sup_{t}\int_{K_{\rho_{n}}}G^{r}\left(\frac{\zeta_n u_{n}}{\rho_n}\right)u_{n}^{s} \zeta_{n}^{q}\,dx \\
&\quad + \iint_{Q_{n}} G\left(|Du_{n}|\right)G^{r-1}\left(\frac{\zeta_n u_{n}}{\rho_n}\right)u_{n}^{s} \zeta_{n}^{q} \,dx\,dt \\
&\leq \gamma 2^{n g_{1}}\left(1+\frac 1\theta\right) G^{r}\left(\frac{k}{\rho_n}\right) k^{s} \left|A_n\right|.
\end{split}\end{equation}

We now consider the function
\begin{equation*}
v = G^{r} \left(\frac{{\zeta}_{n}u_{n}}{2\rho_n}\right) u_{n}^{s} \zeta_{n}^{q}.
\end{equation*}
After taking the derivative of $v$ and applying Lemma~\ref{S1:ineq}, we derive, for some constants $c_0$ and $c_1$, 
\[
|Dv| \leq \frac{c_0}{\rho_{n}}G(|Du_{n}|)G^{r-1}\left(\frac{u_{n}}{2\rho_n}\right) u_{n}^{s} +  \frac{c_1 2^{n}}{\rho_{n}} v .
\]

It follows from \eqref {S3:prop4-E02} that
\[
\sup_t \int_{K_{\rho_n}} v\, dx \le \gamma \left(1+\frac 1\theta\right)2^{n(g_1+2)}k^{s-2} G^{r+1}\left(\frac k{\rho_n}\right) |A_n|
\]
and that
\[
\iint_{Q_n} |Dv|\, dx\,dt \le \gamma\left(1+\frac1\theta\right) \frac 1{\rho_n} 2^{ng_1} k^sG^r\left(\frac k{\rho_n}\right) |A_n|.
\]

Hence, from Theorem~\ref{S4:Embedding-theorem} (and recalling that $\rho/2\le\rho_n\le\rho$), we conclude that
\begin{equation}\label{S3:prop4-E04}\begin{split}
& \iint_{{Q}_{n}} G^{r} \left(\frac{{\zeta}_{n} u_{n}}{\rho_n}\right) {u}_{n}^{s} {\zeta}_{n}^{q}\,dx\,dt \\
&\leq \gamma\left(1+\frac 1\theta\right) 2^{n(g_1+2)} \\
&\quad \times k^{s-2/(N+1)}\rho^{-N/(N+1)} G^{r+1/(N+1)}\left(\frac k{\rho}\right)|A_n|^{(N+2)/(N+1)}.
\end{split}\end{equation}

To find a lower bound for the left hand side of \eqref{S3:prop4-E04}, we observe that in the set $\{u<k_{n+1}\}$, we have
\[
u_n = \max \{0, k_n - u\} \geq k_n - k_{n+1} = \frac{k}{2^{n+2}},\\
\]
It follows that, in $A_{n+1}$, we have
\[
G^r\left( \frac {\zeta_nu_n}{\rho_n}\right) u_n^s\zeta_n^q\ge G^r\left( \frac {k}{2^{n+2}\rho_n}\right)k^s2^{-s(n+2)}
\]
because $\zeta_n=1$ in $Q_{n+1}$.
Since $G^r$ is increasing, we infer that
\[
G^r\left( \frac {\zeta_nu_n}{\rho_n}\right) u_n^s\zeta_n^q \ge 2^{-(s+g_1)(n+2)-g_1}k^s G^r\left(\frac k{\rho}\right)
\]
in $A_{n+1}$, and therefore it follows that
\begin {align*}
|A_{n+1}| &\le \gamma\left(1+\frac 1\theta\right)^{N/(N+1)} 2^{n(2g_1+s+2)} \\
&\quad \times k^{-2/(N+1)}\rho_n^{-N/(N+1)} G^{1/(N+1)}\left(\frac k{\rho}\right)|A_n|^{(N+2)/(N+1)}.
\end {align*}
Hence \eqref {S4:Iteration-lemma:E} is satisfied with
\begin {gather*}
Y_n= |A_n|, \quad C=\gamma\left(1+\frac 1\theta\right)^{N/(N+1)}k^{-\frac 2{N+1}}\rho_n^{-\frac N{N+1}}G^{\frac 1{N+1}}\left( \frac k\rho\right), \\
 b= 2^{2g_1+s+2}, \quad \alpha =\frac 1{N+1}.
\end {gather*}
Applying Lemma~\ref{S4:Iteration-lemma} completes the proof because 
\[
C^{-1/\alpha}=\gamma \left(1+\frac 1\theta\right)^{-N-1}k^2\rho^NG\left(\frac k\rho\right)^{-1} = \gamma \frac {\theta^{N}}{(1+\theta)^{N+1}}|Q_{k,\rho}(\theta)|.
\]
\end{proof}

Note that $\nu_0$ has the form $\nu_1\theta^{N}(1+\theta)^{-N-1}$ with $\nu_1$ determined only by the data.

A variant form of this proposition will also be useful in our study of degenerate equations.

\begin{proposition}\label{S3:prop4a}
For a given positive constant $\theta$,  there exists $\nu_0 = \nu_{0} (\theta, \data) \in (0,1)$ such that, if $u$ is a nonnegative supersolution of \eqref {I:gen} in $Q_{k,2\rho}(\theta)$ with
\begin {subequations} \label {S3:prop4a:SC}
\begin {gather}
\left|\left\{(x,t)\in Q_{k,2\rho}(\theta): u(x,t) < k \right\}\right| < \frac {\nu_0}{\theta} \left|Q_{k,2\rho}(\theta)\right| \\
\intertext {for some positive constants $k$ and $\rho$ and if}
u(x,-T_{k,2\rho}(\theta))\ge k
\end {gather}
\end {subequations}
for all $x\in K_{2\rho}$, then
\[
\essinf_{K_\rho\times(-T_{k,2\rho}(\theta),0)} u \geq \frac{k}{2}.
\]
\end{proposition}

\begin{proof}With $\rho_n$ and $k_n$ as in the proof of Proposition~\ref {S3:prop4}, we set
\[
Q_n=K_{\rho_n}\times (-T_{k,2\rho}(\theta),0),
\]
and we take $\zeta_n$ to be a time-independent cut-off function. In other words,
\[
\zeta_n= \begin {cases} 1&\text { inside } Q_{n+1}, \\
0 &\text { on the lateral boundary of }Q_{n}
\end {cases}
\]
with $\zeta_{n,t}=0$ and $|D\zeta_n|\le 2^{n+1}/\rho_n$.

In place of \eqref {S3:prop4-E01}, we now have
\begin{align*}
    & \sup_{t}\int_{K_{\rho_{n}}} G^{r-1} \left(\frac{\zeta_{n} (u-k_{n})_{-}}{\rho_{n}}\right) (u-k_{n})_{-}^{s+2} \zeta_{n}^{q} \,dx \\
    &\quad + \iint_{Q_{n}} G\left(|D(u-k_{n})_{-}|\right) G^{r-1} \left(\frac{\zeta_{n} (u-k_{n})_{-}}{\rho_{n}}\right) (u-k_{n})_{-}^{s} \zeta_{n}^{q} \,dx\,dt \\
    &\leq \gamma_{1} 2^{(n+1)g_1} \iint_{Q_{n}} G^{r} \left(\frac{\zeta_{n} (u-k_{n})_{-}}{\rho_{n}}\right) (u-k_{n})_{-}^{s}  \,dx\,dt.
\end{align*}
Arguing as in the proof of Proposition~\ref {S3:prop4}, we now infer that \eqref {S4:Iteration-lemma:E} is satisfied with
\begin {gather*}
Y_n= |A_n|, \quad C=\gamma k^{-\frac 2{N+1}}\rho_n^{-\frac N{N+1}}G^{\frac 1{N+1}}\left( \frac k\rho\right), \\
 b= 2^{2g_1+s+2}, \quad \alpha =\frac 1{N+1}.
\end {gather*}
The proof is completed by noting that
\[
C^{-1/\alpha} = \frac {\gamma}\theta |Q_{k,\rho}(\theta)|\ge \frac \gamma\theta|Q_{k,2\rho}(\theta)|.
\]
\end {proof}

\subsection {An improvement of Proposition~\ref {S3:prop3}}

For our study of singular equations, we need a stronger result than Proposition~\ref {S3:prop3}.  
Throughout this subsection, $\nu$, $\nu_0$, $\rho$, and $k$ are given positive constants with $\nu, \nu_0<1$.
Also, to simplify notation, we set
\[
T=\left( \frac k2\right)^2 G\left( \frac {k}{2\rho}\right)^{-1}.
\]

We assume that $u$ is a nonnegative supersolution of
 \begin {equation} \label {E:51}
 u_t=\di A(x,t,u,Du) \text { in } K_{2\rho}\times (-T ,0)
 \end {equation}
 and that
\begin {equation} \label {E:51large}
\left| \left\{x\in K_{2\rho}: u(x,t) \le \frac k2\right\}\right| 
\le (1-\nu_0)|K_{2\rho}|
\end {equation}
for all $t\in (-T,0)$.

We wish to prove the following proposition, which is a generalization of Lemma 5.1 from Chapter IV of \cite {DB93}. In fact, this lemma is not the complete first alternative as described in that source; we single it out as the crucial step in that alternative.

\begin {proposition} \label {S3:L1}
If $g_1\le2$ and if $u$ is a nonnegative supersolution of \eqref {E:51} which satisfies \eqref {E:51large}, then there is a constant $\delta^*$ determined only by $\nu$, $\nu_0$, and the data such that
\begin {equation} \label {S3:L1estimate}
\left | \left\{ x\in K_\rho: u(x,t)\le \frac {\delta^*k}2\right\}\right| \le \nu |K_{2\rho}|
\end {equation}
for all $t\in (-T_1,0)$, where
\begin {equation} \label {S3:L1T1}
T_1=\left( \frac k2\right)^2 G\left( \frac {k}{\rho}\right)^{-1}.
\end {equation}
\end {proposition}
 
Our proof follows that of Lemma 5.1 from Chapter IV of \cite {DB93} rather closely with a few modifications based on ideas from Section 4 of \cite {Lie91}. In addition, our proof shows more easily that the constants in Chapter IV of \cite {DB93} are stable as $p\ne2$.

Our first step is as in Section 6 from Chapter IV of \cite {DB93}. We show that $u$ satisfies an additional integral inequality, which is the basis of the proof of Proposition \ref {S3:L1}.
Before stating our inequalities, we introduce some notation. For positive constants $\kappa$ and $\delta$ with $\kappa\le k/2$ and $\delta<1$, we define two functions $\Phi_\kappa$ and $\Psi_\kappa$ as follows:
\begin {subequations} \label {S3:kfunctions}
\begin {align}
\Phi_\kappa(\sigma) &= \int_0^{(\kappa-\sigma)_+} \frac {(1+\delta)\kappa-s}{G\left( \frac {(1+\delta)\kappa-s}{2\rho}\right)}\, ds, \\
\Psi_\kappa(\sigma) &= \ln\left[ \frac {(1+\delta)\kappa}{(1+\delta)\kappa-(\kappa-\sigma)_+}\right].
\end {align}
\end {subequations}
We also note that there are two Lipschitz functions, $\zeta_1$ defined on $K_{2\rho}$ and $\zeta_2$ defined on $[-T,0]$ such that
\begin {subequations} \label {S3:zeta}
\begin {gather}
\zeta_1 =0 \text { on the  boundary of }K_{2\rho}, \\
\zeta_1 =1 \text { in } K_\rho, \\
|D\zeta_1| \le \frac 1\rho \text { in } K_{2\rho}, \\
\{x\in K_{2\rho}: \zeta_1(x)>\varepsilon\} \text { is convex for all }\varepsilon\in(0,1), \\
\zeta_2(-T)=0, \\
\zeta_2=1 \text { on } (-T_1,0), \\
0 \le \zeta_2' \le  \left( \frac 2k\right)^2G\left(\frac k \rho\right)  \text { on } (-T_1,0). \label {S3:zetat} 
\end {gather}
\end {subequations}
Let us note that it's easy to arrange that $\zeta_2'\ge0$ and that
\[
\frac 1{\zeta_2'}\ge \left(\frac k2\right)^2G\left(\frac k{2\rho}\right)^{-1}-\left(\frac k2\right)^2G\left(\frac {k}{\rho}\right)^{-1}.
\]
Since Lemma~\ref {S1:ineq}(b) implies that
\[
G\left(\frac {k}{\rho}\right) \ge 2^{g_0}G\left(\frac {k}{2\rho}\right) \ge 2G\left(\frac {k}{2\rho}\right),
\]
we infer the second inequality of \eqref {S3:zetat}.

Also, we introduce the notation $D^-$ to denote the derivate
\[
D^-f(t)= \limsup_{h\to0^+} \frac {f(t)-f(t-h)}h.
\]

With these preliminaries, we can now state our integral inequality.

\begin {lemma} \label {S3:integralinequality} 
If $g_1\le2$ and if $u$ is a weak supersolution of \eqref {E:51} satisfying \eqref {E:51large}, then there are positive constants $\gamma$ and $\gamma_0$, determined only by $\nu$, $\nu_0$, and the data such that
\begin {equation} \label {S3:Eintegralinequality}
D^-\left( \int_{K_{2\rho}} \Phi_\kappa(u(x,t))\zeta^q(x,t)\, dx\right)+ \gamma_0\int_{K_{2\rho}} \Psi_\kappa^{g_0}(u(x,t))\zeta^q(x,t)\, dx \le \gamma |K_{2\rho}|
\end {equation}
for all $t\in (-T,0)$, where 
\begin {equation} \label {S3:integralinequalityq}
q= g_0/(g_0-1).
\end {equation}
\end {lemma}
\begin {proof}
With
\[
u^*=\frac {(1+\delta)\kappa-(\kappa-u)_+}{2\rho},
\]
we use the test function
\[
\frac {\zeta^q((1+\delta)\kappa-(\kappa-u)_+)}{G(u^*)}
\]
in the weak form of the differential inequality satisfied by $u$ to infer that, for all sufficiently small positive $h$, we have
\[
I_1 +I_2\le I_3+I_4
\]
with
\begin {align*}
I_1 &=\left(\int_{K_{2\rho}} \Phi_\kappa(u(x,t))\zeta^q(x,t)\,dx-\int_{K_{2\rho}} \Phi_\kappa(u(x,t-h))\zeta^q(x,t-h)\,dx\right), \\
I_2 &=\int_{t-h}^h\int_{K_{2\rho}} \zeta^q(x,\tau) D(\kappa-u)_+(x,\tau)\cdot A \frac 1{G(u^*(x,\tau))}\left[1- \frac{u^*(x,\tau) g(u^*(x,\tau))}{G(u^*(x,\tau))}\right]\, dx\,d\tau, \\
 I_3 &=q\int_{t-h}^t \int_{K_{2\rho}} D\zeta(x,\tau) \cdot A  \zeta^{q-1}(x,\tau) \frac {(1+\delta)\kappa-(\kappa-u)_+}{G\left(u^*(x,\tau)\right)} \,dx\, d\tau,\\
 I_4&=q\int_{t-h}^t \int_{K_{2\rho}} \Phi_\kappa(u(x,\tau)) \zeta^{q-1}(x,\tau)\zeta_t(x,\tau)\,dx\,d\tau,
\end {align*}
and $A$ evaluated at $(x,\tau,u(x,\tau),Du(x,\tau))$ in $I_2$ and $I_3$.
We now use (0.3a) and the first inequality in (0.4) to see that
\[
I_2 \ge C_0(g_0-1) \int_{t-h}^t \int_{K_{2\rho}} \zeta^q(x,\tau) \frac {G(|D(\kappa-u)_+(x,\tau)|)}{G(u^*(x,\tau))}\,dx\,d\tau.
\]
Also, (0.3b) and Lemma~1.1(e) (with $\sigma_1= (qC_1/C_0)|D\zeta(x,\tau)|\rho u^*(x,\tau)$, $\sigma_2= |D(\kappa-u)_+(x,\tau)|$, and $\varepsilon = \zeta(x,\tau) (g_0-1)/(2g_1)$) imply that
\[
qD\zeta(x,\tau) \cdot A (x,\tau,u,Du) \zeta^{q-1}(x,\tau) \frac {(1+\delta)\kappa-(\kappa-u)_+}{G\left(u^*(x,\tau)\right)} \le J_1+J_2
\]
with
\[
J_1 = g_1\left( \frac {2g_1}{(g_0-1)}\right)^{g_1-1}\zeta^{q-g_1}(x,\tau)\frac {G(q(C_1/C_0)|D\zeta|\rho u^*)}{G(u^*)}
\]
and
\[
J_2 = \frac 12 C_0(g_0-1)\zeta^q \frac {G(|D(\kappa -u)_+|)}{G(u^*)}.
\]
From our conditions on $\zeta$ and because $q\ge2\ge g_1$, we conclude that there is a constant $\gamma_1$, determined only by data, such that
\[
J_1 \le \gamma_1,
\]
so
\[
I_3 \le \gamma _1h|K_{2\rho}|+ \frac 12I_2.
\]

Next, we estimate $\Phi_\kappa$.  Since $\kappa\le k/2$ and $\delta\in (0,1)$, it follows that, for all $s\in (0,(\kappa-u)_+)$, we have $(1+\delta)\kappa-s \le 2k$ and hence
\[
G\left( \frac {(1+\delta)\kappa-s}{2\rho}\right) \ge \left(\frac {(1+\delta)\kappa-s}{2k}\right)^2 G\left(\frac k \rho\right).
\]
It follows that
\begin {align*}
\Phi_\kappa(u) &\le 4k^2G\left(\frac k \rho\right)^{-1} \int_0^{(\kappa-u)_+} [(1+\delta)\kappa-s]^{-1}\, ds \\
&=4k^2G\left(\frac k \rho\right)^{-1} \Psi_\kappa(u),
\end {align*}
and therefore
\[
I_4\le 16q\int_{t-h}^t\int_{K_{2\rho}}\Psi_\kappa(u(x,\tau))\zeta^{q-1}(x,\tau)\, dx\, d\tau.
\]

Combining all these inequalities and setting
\begin {align*}
I_{21} &= \int_{t-h}^t \int_{K_{2\rho}} \zeta^2(x,\tau) \frac {G(|D(\kappa-u)_+(x,\tau)|)}{G(u^*(x,\tau))}\,dx\,d\tau, \\
I_{41} &= \int_{t-h}^t\int_{K_{2\rho}}\Psi_\kappa(u(x,\tau))\zeta^{q-1}(x,\tau)\, dx\, d\tau
\end {align*}
yields
\begin {equation} \label {S3:Lintegralinequality1}
I_1+\frac 12C_0(g_0-1) I_{21} \le  \gamma_1 h|K_{2\rho}| + 16qI_{41}.
\end {equation}
Our next step is to compare
\[
I_{22} = \int_{t-h}^t\int_{K_{2\rho}}\zeta^q(x,\tau) \Psi_\kappa^{g_0}(u(x,\tau))\,dx\,d\tau
\]
to $I_{21}$. To this end, we first use Lemma~3.1 with $\varphi=\zeta_1^q$, $v=(\kappa-u)_+$, and $p=g_0$ to conclude that there is a constant $\gamma_2$ determined only by the data and $\nu_0$ such that, for almost all $\tau \in (t-h,t)$, we have
\begin {equation} \label {S3:I21}
\int_{K_{2\rho}} \zeta^q(x,\tau)\Psi_\kappa^{g_0}(u(x,\tau))\, dx \le \gamma_2\rho^{g_0}\int_{K_{2\rho}}\zeta^q(x,\tau)|D\Psi_\kappa(u(x,\tau))|^{g_0}\, dx.
\end {equation}
(Of course, we have multiplied \eqref {Lpoin:E} by $\zeta_2^q(\tau)$ here.)
Now we use the explicit expression for $\Psi_\kappa$ to infer that
\[
\rho|D\Psi_\kappa(u(x,\tau))| = \frac {|D(\kappa-u)_+(x,\tau)|}{2u^*(x,\tau)} \le\frac {|D(\kappa-u)_+(x,\tau)|}{u^*(x,\tau)}.
\]
Whenever $|D(\kappa-u)_+(x,\tau)| \le u^*(x,\tau)$, we conclude that
\[
\rho^{g_0}|D\Psi_\kappa(u(x,\tau))|^{g_0}\le 1
\]
and, wherever $|D(\kappa-u)_+(x,\tau)| > u^*(x,\tau)$, we infer from Lemma~1.1 that
\[
\rho^{g_0}|D\Psi_\kappa(u(x,\tau))|^{g_0}\le \frac {G(|D(\kappa-u)_+(x,\tau)|)}{G(u^*(x,\tau))}.
\]
It follows that, for any $(x,\tau)$, we have
\[
\rho^{g_0}|D\Psi_\kappa(u(x,\tau))|^{g_0}\le 1+\frac {G(|D(\kappa-u)_+(x,\tau)|)}{G(u^*(x,\tau))}.
\]
Inserting this inequality into \eqref {S3:I21} and integrating the resultant inequality with respect to $\tau$ yields
\[
I_{22} \le \gamma_2(I_{21}+h|K_{2\rho}|).
\]
By invoking \eqref {S3:Lintegralinequality1}, we conclude that
\[
I_1+ \frac {1}{2\gamma_2}C_0(g_0-1)I_{22} \le \left(\gamma_1+ \frac {1}{2\gamma_2}C_0(g_0-1)\right)h|K_{2\rho}| +16qI_{41}.
\]
We now note that
\[
\Psi_\kappa(u)\zeta^{q-1}= \left(\Psi_\kappa^{g_0}(u) \zeta^q\right)^{1/g_0},
\]
so Young's inequality shows that
\[
\Psi_\kappa(u)\zeta^{q-1} \le \varepsilon \Psi_\kappa^{g_0}(u) \zeta^q + \varepsilon^{-q}
\]
for any $\varepsilon\in(0,1)$.
By choosing $\varepsilon$ sufficiently small, we see that there are constants $\gamma_0$ and $\gamma$ such that
\[
I_1+\gamma_0 I_2 \le \gamma h|K_{2\rho}|.
\]
To complete the proof, we divide  this inequality by $h$ and take the limit superior as $h\to0^+$.
\end {proof} 

Our proof is essentially the same as that for Lemma 6.1 of Chapter IV from \cite {DB93}; the new ingredient is a more careful estimate of the integral involving $\zeta_t$ (which we have denoted by $I_4$).  In this way, we obtain an estimate which does not depend on $p-2$ being bounded away from zero, which was the case in (6.9) of Chapter IV from \cite {DB93}.

Our next step is to estimate the integral of $\zeta^q$ over suitable $N$-dimensional sets with $q$ defined by \eqref {S3:integralinequalityq}.  Specifically, for each positive integer $n$ and a number $\delta\in(0,1)$ to be further specified, we define the set
\[
K_{\rho,n}(t) = \{x\in K_{2\rho}: u(x,t)<\delta^nk\}
\]
and we introduce the quantities
\begin {align*}
A_n(t) &= \frac 1{|K_{2\rho}|} \int_{K_{\rho,n}(t)} \zeta^q(x,t)\, dx, \\
Y_n &=  \sup_{-T<t<0}A_n(t)
\end {align*}
We shall show that, for a suitable choice of $\delta$ (which will require at least that $\delta\le1/2$) and $n$, we can make $Y_n$ small.  In fact, based on the discussion in Section 7 of Chapter IV from \cite {DB93}, we shall find $n_0$ and $\delta$ so that $Y_{n_0} \le \nu$.  In fact, our method is to estimate $A_{n+1}(t)$ in terms of $Y_n$ for each $n$.

We first estimate $A_{n+1}(t)$ if
\begin {equation} \label {S3:Dge}
D^- \left( \int_{K_{2\rho}} \zeta^q(x,t) \Phi_{\delta^nk} (u(x,t))\, dx\right) \ge 0.
\end {equation}
(This is the case (7.5) of Chapter IV from \cite {DB93}.)  Our estimate now takes the following form.

\begin {lemma} \label {S3:L7.5}
If \eqref {S3:Dge} holds, then, for all $\nu$ and $\nu_0$ in $(0,1)$, there is a constant $\delta_0$, determined only by $\nu$, $\nu_0$, and the data, such that $\delta\le \delta_0$ implies that
\begin {equation} \label {S3:L7.5:Yn1}
A_{n+1}(t) \le \nu.
\end {equation}
\end {lemma}
\begin {proof}
On $K_{\rho,n+1}(t)$, we have
\begin {align*}
\Psi_{\delta^nk}(u) &= \ln \left[ \frac {(1+\delta)\delta^nk}{(1+\delta)\delta^nk - (\delta^nk-u)_+} \right]\\
& \ge \ln \left[ \frac {(1+\delta)\delta^nk}{(1+\delta)\delta^nk - (\delta^nk-\delta^{n+1}k)_+} \right] \\
&= \ln  \frac {1+\delta}{2\delta}.
\end {align*}
It follows that
\begin {align*}
\left( \ln \frac {1+\delta}{2\delta} \right)^{g_0} \int_{K_{\rho,n+1}(t)}&\zeta^q(x,t)\, dx \\
& \le \int_{K_{\rho,n+1}(t)} \zeta^q(x,t) \Psi_{\delta^nk}(u(x,t))\, dx.
\end {align*}
By invoking \eqref {S3:Eintegralinequality} and \eqref {S3:Dge}, we conclude that
\[
\int_{K_{\rho,n+1}(t)}\zeta^q(x,t)\, dx\le \frac \gamma{\gamma_0} \left( \ln \frac {1+\delta}{2\delta} \right)^{-g_0}|K_{2\rho}|.
\]
By choosing $\delta_0$ sufficiently small, we infer \eqref {S3:L7.5:Yn1}.
\end {proof}

Our estimate when \eqref {S3:Dge} fails is more complicated, as shown for the power case in Section 8 from Chapter IV of \cite {DB93}.

\begin {lemma} \label {S3:L7.6}
For all $\nu$ and $\nu_0$ in $(0,1)$, there are positive constants $\delta_1$ and $\sigma<1$, determined only by $\nu$, $\nu_0$, and the data, such that if
\begin {equation} \label {S3:Dless}
D^-\left( \int_{K_{2\rho}} \zeta^q(x,t)\Phi_{\delta^nk} (u(x,t))\, dx \right)<0
\end {equation}
for some  $\delta\in(0,\delta_1)$ and if $Y_n>\nu$, then 
\begin {equation} \label {S3:L7.6equation}
A_{n+1}(t) \le \sigma Y_n.
\end {equation}
\end {lemma}
\begin {proof} In this case, we define
\[
t_*=\sup\left\{\tau\in(-T,t): D^-\left( \int_{K_{2\rho}} \zeta^q(x,\tau)\Phi_{\delta^nk} (u(x,\tau))\, dx \right) \ge0 \right\}
\]
and note that this set is nonempty.  From the definition of $t_*$, we have that
\begin {equation} \label {S3:L7.6:Phiint}
\int_{K_{2\rho}} \zeta^q(x,t) \Phi_{\delta^nk} u(x,t)\, dx \le \int_{K_{2\rho}} \zeta^q(x,t_{*}) \Phi_{\delta^nk} u(x,t_{*})\, dx.
\end {equation}
It follows from Lemma~\ref {S3:integralinequality} and the definition of $t_*$ that
\[
\int_{K_{2\rho}} \zeta^q(x,t_{*}) \Psi^{g_0}_{\delta^nk} u(x,t_{*})\, dx\le C|K_{2\rho},
\]
with
\[
C= \frac \gamma{\gamma_0}.
\]
Now we set
\[
K_*(s)=\{x\in K_{2\rho}: (\delta^nk -u)_+(x,t_*) >s\delta^nk\}
\]
for $s\in (0,1)$ and
\[
I_1(s)= \int_{K_*(s)} \zeta^q(x,t_*)\, dx.
\]
As in the proof of Lemma~\ref {S3:L7.5}, we have that
\[
\Psi_{\delta^nk}(u(x,t_*)) \ge \ln\frac {1+\delta}{1+\delta-s}
\]
on $K_*(s)$, so
\begin {equation} \label {S3:L7.6:1}
I_1(s)\le C\left(\ln\frac {1+\delta}{1+\delta-s}\right)^{-g_0}|K_{2\rho}|.
\end {equation}
Moreover, if $x\in K_*(s)$, then 
\[
u(x,t_*) <(1-s)\delta^nk \le \delta^nk,
\]
and hence $K_*(s)\subset K_{\rho,n}(t_*)$, so
\begin {equation} \label {S3:L7.6:2}
I_1(s)\le Y_n|K_{2\rho}|.
\end {equation}
We now define 
\[
s_*= \left[1-\exp\left(- \left( \frac {2C}{\nu}\right)^{1/g_0}\right)\right](1+\delta_*), 
\]
with $\delta_*\in(0,1)$ chosen so that $s_*<1$.
Since $Y_n>\nu$, a simple calculation shows that
\begin {equation} \label {S3:L7.6:CYn}
C\left(\ln \frac {1+\delta}{1+\delta -s}\right)^{-g_0}\le \frac 12Y_n
\end {equation}
 for $s>s_*$ provided $\delta\le \delta_*$.

Next, we set
\[
I_2 =\int_{K_{2\rho}} \zeta^q(x,t_*)\Phi_{\delta^nk}(u(x,t_*))\, dx
\]
and use Fubini's theorem to conclude that
\begin {align*}
I_2 &=\int_{K_{2\rho}} \zeta^q(x,t_*)  \left( \int_0^{\delta^nk} \frac { \chi_{\{(\delta^nk-u)+>s \}}((1+\delta)\delta^nk-s) }{G\left( \frac {(1+\delta)\delta^nk-s}{2\rho}\right)}\, ds\right)\, dx \\
&= \int_0^{\delta^nk} \frac {(1+\delta)\delta^nk-s}{G\left( \frac {(1+\delta)\delta^nk-s}{2\rho}\right)}\left( \int_{K_{2\rho}} \zeta^q(x,t_*)  \chi_{\{(\delta^nk-u)+>s \}}\, dx\right)\, ds
\end {align*}
Using the change of variables $\tau= s/(\delta^{n}k)$, we see that
\begin {align*}
I_2
&=\int_0^1 \frac {(1+\delta)-\tau}{G\left( \frac {\delta^nk(1+\delta-\tau)}{2\rho}\right)}\left( \int_{K_{2\rho}} \zeta^q(x,t_*)  \chi_{\{(\delta^nk-u)+>\delta^nk\tau \}}\, dx\right)\, d\tau \\
&= \int_0^1 \frac {(1+\delta)-\tau}{G\left( \frac {\delta^nk(1+\delta-\tau)}{2\rho}\right)}I_1(\tau)\, d\tau.
\end {align*}
Combining this equation with \eqref {S3:L7.6:1}, \eqref {S3:L7.6:2}, and \eqref {S3:L7.6:CYn} then yields
\[
I_2\le Y_n|K_{2\rho}|\left[\int_0^{s_*} \frac {(1+\delta)-\tau}{G\left( \frac {\delta^nk(1+\delta-\tau)}{2\rho}\right)}\, d\tau + \frac 12\int_{s_*}^1 \frac {(1+\delta)-\tau}{G\left( \frac {\delta^nk(1+\delta-\tau)}{2\rho}\right)}\, d\tau\right].
\]
We now define the function $f$ by 
\[
f(\tau)= \frac {\tau}{G\left( \frac {\delta^nk\tau}{2\rho}\right)}
\]
and we set $\sigma_*= 1-s_*$. Using the change of variables $s=1-\tau$ then yields
\[
I_2 \le Y_n|K_{2\rho}|\mathcal K,
\]
with
\[
\mathcal K= \int_{\sigma_*}^1 f(\delta+s)\, ds + \frac 12\int_0^{\sigma_*}f(\delta+s)\, ds.
\]
Since
\[
\mathcal K = \int_0^1f(\delta+s)\, ds -\frac 12 \int_0^{\sigma_*}
 f(\delta+s)\, ds,
 \]
it follows from Lemma~\ref {S1:fintegral} (specifically \eqref {S1:fintegral:E1}) that 
\[
\mathcal K \le (1-\frac {\sigma_*}2) \int_0^1f(\delta+s)\, ds,
\]
and therefore
\begin {equation} \label {S3:L7.6:Phi1}
I_2 \le Y_n|K_{2\rho}|\left(1-\frac {\sigma_*}2\right)  \int_0^1f(\delta+s)\, ds.
\end {equation}

Our next step is to infer a lower bound for $I_2$. Taking into account \eqref {S3:L7.6:Phiint}, we have
\[
I_2 \ge \int_{K_{\rho,n+1}(t)} \zeta^q(x,t)\Phi_{\delta^nk}(u(x,t))\, dx.
\]
Next, for $z<\delta^{n+1}k$, we have
\begin {align*}
\Phi_{\delta^nk}(z) &= \int_0^{(\delta^nk-z)_+}\frac {(1+\delta)\delta^nk -s} {G\left( \frac {(1+\delta)\delta^nk-s}{2\rho}\right)}\, ds \\
&\ge \int_0^{\delta^nk(1-\delta)}\frac {(1+\delta)\delta^nk -s} {G\left( \frac {(1+\delta)\delta^nk-s}{2\rho}\right)}\, ds \\
&= \int_0^{1-\delta} f(\delta+s)\, ds \\
&\ge \left(1- \frac 2{2+\ln(1/\delta)}\right) \int_0^1 f(\delta+s)\,ds
\end {align*}
by virtue of \eqref {S1:fintegral:E2}, so
\[
\Phi_{\delta^nk}(u(x,t)) \ge \left(1- \frac 2{2+\ln(1/\delta)}\right)\int_0^\delta f(\delta+s)\, ds
\]
for all $x\in K_{\rho,n+1}(t)$ and hence
\[
I_2 \ge \left(1- \frac 2{2+\ln(1/\delta)}\right)\left(\int_{K_{\rho,n+1}(t)} \zeta^q(x,t)\, dx\right)\left(\int_0^\delta f(\delta+s)\, ds\right).
\]
In conjunction with \eqref {S3:L7.6:Phi1}, this inequality implies that
\[
A_{n+1}(t) \le \frac {1-(\sigma_*/2)}{1- ( 2/(2+\ln(1/\delta))}Y_n.
\]
By taking $\delta_2$ sufficiently small, we can make sure that
\[
\sigma=\frac {1-(\sigma_*/2)}{1- ( 2/(2+\ln(1/\delta_2))}
\]
is in the interval $(0,1)$. If we take $\delta_1=\min\{\delta_*,\delta_2\}$, we then infer \eqref {S3:L7.6equation} for $\delta\le\delta_1$.
 \end {proof}
 
We are now ready to prove Proposition~\ref {S3:L1}.

\begin {proof}[Proof of Proposition~\ref {S3:L1}]

Since $Y_{n+1} \le Y_n$, it follows from Lemmata~\ref {S3:L7.5} and \ref {S3:L7.6} that, for all positive integers $n$, we have
\[
A_{n+1}(t) \le \max \{\nu,\sigma Y_n\}
\]
for all $t\in (-T,0)$ and hence
\[
Y_{n+1}\le \max\{\nu,\sigma Y_n\}
\]
Induction implies that
\[
Y_n \le \max\{\nu,\sigma^{n-1} Y_1\}
\]
for all $n$. In addition $Y_1\le1$, so there is a positive integer $n_0$, determined by $\nu$, $a_0$, and the data such that $Y_{n_0}\le\nu$.

Next, we recall that $\zeta=1$ on $K_\rho\times(-T_1,0)$, and hence, for all $t\in(-T_1,0)$, we have 
\begin {align*}
\left| \left\{ x\in K_\rho:u(x,t)\le \delta^{n_0}k\right\}\right| &= \int_{\{x\in K_\rho: u(x,t)\le \delta^{n_0}k\} }\zeta^q(x,t)\, dx \\
&\le \int_{\{x\in K_{2\rho}: u(x,t)\le \delta^{n_0}k\}} \zeta^q(x,t)\, dx \\
&\le Y_{n_0}.
\end {align*}
The proof is completed by using the inequality $Y_{n_0} \le \nu$ and taking $\delta^*=\delta^{n_0}$.
\end {proof}

\subsection {Proof of the Main Lemma for singular equations}

\begin {proof}
With $\delta$ to be chosen, we use Proposition~ \ref {S3:prop1} with $k=\omega/2$, $\rho=2R$, $\nu_1=\frac 12$, and 
\[
T= 3\left(\frac {\delta \omega}2\right)^2G\left( \frac {\delta \omega}{2R}\right)^{-1}
\]
to infer that there is a $\tau_1\in (-T, -T/3)$ such that
\[
\left|\left\{ x\in K_{2R}: u(x,\tau_1)\le \frac \omega2\right\}\right|\le \frac 34|K_{2R}|.
\]

Next, we set $\nu=\frac 18$, $\rho=2R$, $k=\omega/2$, $\tau=\tau_1$, and $\theta=3$. Since $\tau_1\le T$ and
\begin {align*}
T &=3\left(\frac {\delta \omega}2\right)^2G\left( \frac {\delta \omega}{2R}\right)^{-1} \\
&= 3(\delta k)^2 G\left( \frac {2\delta k}\rho\right)^{-1} \\
 &\le 3(\delta k)^2G\left( \frac {\delta k}{\rho}\right)^{-1},
 \end {align*}
 it follows that \eqref {S3:prop2-hyp} is satisfied, so Proposition~ \ref {S3:prop2} implies that
 \begin {equation} \label {S2:MainLemmaE1}
 \left| \left\{x\in K_{2R}: u(x,t)\le \frac {\delta \omega}2\right\}\right| \le \frac 78|K_{2R}|
 \end {equation}
 for all $t\in (\tau_1,0)$ provided we take $\delta$ to be the constant from that proposition. (In particular, $\delta$ is determined only by the data.)
 Since $\tau_1\ge T/3$, it follows that
 \[
 \tau_1 \ge \left(\frac {\delta \omega}2\right)^2G\left( \frac {\delta \omega}{2R}\right)^{-1},
 \]
 and hence \eqref {S2:MainLemmaE1} holds for all
 \[
 t\in\left( -\left(\frac {\delta \omega}2\right)^2G\left( \frac {\delta \omega}{2R}\right)^{-1},0\right).
 \]
 
 Now we use Proposition~\ref {S3:L1}, with $\omega=\delta \omega$ and $\nu$ to be chosen, to infer that there is a constant $\delta^*\in(0,1)$, determined only by the data and $\nu$ such that
\begin {equation} \label {S2:MainLemmaE2}
\left|\left\{ x\in K_R: u(x,t) \le \frac {\delta^*\delta \omega}2\right\} \right| \le \nu |K_{2R}|
\end {equation}
for all
\[
t\in \left( - \left( \frac {\delta \omega}2\right)^2G\left( \frac {\delta \omega}R\right)^{-1}, 0\right).
\]
Since
\[
G\left( \frac {\delta \omega/2}R\right) \le 16G\left(\frac {\delta \omega}R\right),
\]
and $\delta^*\le1$, it follows that
\begin {align*}
\left( \frac {\delta \omega}2\right)^2G\left( \frac {\delta \omega}R\right)^{-1} &\ge \frac 1{16}\left( \frac {\delta \omega}2\right)^2G\left( \frac {\delta \omega/2}R\right)^{-1} \\
&\ge \frac 1{16}\left( \frac {\delta^*\delta \omega/2}R\right)^2G\left( \frac {\delta^*\delta \omega}R\right)^{-1}.
\end {align*}

We now take $\nu_0$ to be the constant corresponding to $\theta=1/16$ in Proposition~\ref {S3:prop4}, and we set $\nu=2^{-N}\nu_0$, which determines $\delta^*$. Then \eqref {S3:prop4:hyp} is satisfied for $k=\delta^*\delta \omega/2$ and $\rho=R/2$. Proposition~\ref {S3:prop4} then yields \eqref {S2:MainLemma:estimate} with $\mu =\delta^*\delta/4$.
\end {proof}
\subsection {Proof of the first alternative for degenerate equations}

\begin {proof}
First, we set
\[
T_1=T_0+\Delta-\omega^2G\left( \frac \omega R\right)^{-1}
\]
and we use Proposition~\ref{S3:prop4} to infer that
\[
u\ge \frac \omega2 \text { on } K_R\times\{T_1\}.
\]
It then follows from Proposition~\ref {S3:prop2} with $\nu=1$ and $k=\omega/2$ that, for any $\varepsilon\in(0,1)$, there is a constant $\delta\in(0,1)$ determined only by data, $\varepsilon$, and $\theta_0$ such that
\begin {equation} \label {S5:L1:eps}
|\{x\in K_R: u(x,t)<\frac {\delta\omega}2\}| <\varepsilon|K_R|
\end {equation}
for all $t\in[T_1,0)$.
We now choose $\varepsilon= \nu_0/(2^N\theta_0)$ and set
\[
\theta =\frac {T_1}{(\delta\omega)^2G(\delta\omega/(2R))^{-1}}.
\]
We first observe that
\[
G\left(\frac {\delta \omega}{2R}\right)\le \delta^{g_0}G\left(\frac \omega{2R}\right)
\]
and hence
\[
\theta \le \frac {\theta_0\omega^2 G\left( \frac {\omega}{2R}\right)^{1}}{\delta^{2-g_0}\omega^2 G\left(\frac \omega{2R}\right)^{-1}}= \delta^{g_0-2}\theta_0\le\theta_0.
\]
For $k=\delta\omega/2$, and $\rho=R/2$, we now have that \eqref {S3:prop4a:SC} holds.  The proof is completed by noting that $T_1\le -\omega^2G(\omega/R)^{-1}$. 
\end {proof}

\subsection {Proof of the second alternative}
\begin {proof}
First, we set
\[
\alpha= \frac {\nu_0}{2-\nu_0},
\]
and we apply Proposition~\ref {S3:prop1} with $\theta=1$, $\nu_1=\nu_0$, $k=\omega$, and $\rho=2R$ to infer that, for each $T_0\in(-\theta_0\delta,-\Delta)$, there is a number $\tau_1\in (T_0,T_0+\alpha \Delta)$ such that
\[
\left| \{x\in K_{2R}: u(x,\tau_1)\ge \omega\}\right| \le \left(1-\frac {\nu_0}2\right) |K_{2R}|.
\]

We momentarily fix $T_0\in(-\theta_0\Delta,-\Delta)$.  It follows from Proposition~\ref {S3:prop2} with $\rho=2R$, $\nu=\nu_0/2$, $\varepsilon=\frac 12$, and $\theta=2^{g_1}$ that there is a constant $\delta\in(0,1)$, determined only by data, such that
\begin {equation} \label {S5:L1:delta}
\left| \{x\in K_{2R}: u(x,t)\le \delta\omega\}\right| \le \left(1-\frac {\nu_0}4\right) |K_{2R}|
\end {equation}
for all $t\in (\tau_1,\min\{\tau_1+\theta\Delta',0\})$, where $\Delta'= \omega^2G(\omega/R)^{-1}$. Here we note that
\[
G\left(\frac {\omega}R\right)\le 2^{g_1}G\left(\frac {\omega}{2R}\right)
\]
and hence $\theta \Delta' \ge \Delta$, so \eqref {S5:L1:delta} holds for all $t\in (\tau_1,\tau_1+\Delta)$.
With this $\delta$, it follows that \eqref {S5:L1:delta} holds for all $t\in (T_0+(1-\alpha)\Delta, T_0+\Delta)$.
Finally, since $T_0$ is arbitrary, we conclude that \eqref {S5:L1:delta} holds for all $t\in ((-\theta_0+1)\Delta,0)$.

For our next step, we take $\delta^*$ to be the constant from Proposition~\ref {S3:prop3} corresponding to $\alpha=\nu_0/4$, $\nu=\nu_0$ and $\theta=1$. We also set
\[
\theta_0=1+(\delta^*\delta)^{2-g_1}.
\]
Since
\[
(\theta_0-1)\Delta \ge (\delta^*\delta\omega)^2G\left(\frac {\delta^*\delta\omega}{2R}\right)^{-1},
\]
we infer from Proposition~\ref {S3:prop3} that
\[
\left|Q_{\delta^*\delta\omega,2R}\cap \{u<\delta^*\delta\omega\}\right| \le \nu_0 \left|Q_{\delta^*\delta\omega,2R}\right|
\]
(for $Q_{\delta^*\delta\omega,2R}$ defined with $\theta=1$).
We then use  Proposition~\ref {S3:prop4} with $\theta=1$, $k=\delta^*\delta\omega$, and $\rho=R$ to infer that
\[
\essinf_{K_R\times(-(\delta^*\delta)^2\omega G\left(\frac {\delta^\delta\omega}{2R}\right)^{-1},0)} u \ge \frac 12\delta^*\delta\omega.
\]
The proof is now completed by observing that
\[
-(\delta^*\delta)^2\omega^2 G\left(\frac {\delta^*\delta\omega}{2R}\right)^{-1}\ge \omega^2 G\left(\frac \omega R\right)^{-1}
\]
and that $R\ge R/2$.
\end {proof}

\section{Proof for Auxiliary Theorems}\label{S4}

We now present the basic results used in the previous sections of the paper.  Some are proved here because their proofs are slightly different from the corresponding results for the parabolic $p$-Laplacian equation, but the others, which are already known in the form we need, are just quoted.

    \subsection{The local energy estimate}

    The local energy estimate is one of fundamental inequality playing important roles, especially Proposition~\ref{S3:prop1}, Proposition~\ref{S3:prop2}, and Proposition~\ref{S3:prop3}. The inequality is equivalent to one appearing on Section II.3-(i) from \cite{DB93} and same as Proposition 2.4 from \cite{Urb08} if $g_0 = g_1 = p$. Some techniques come from Section 3 in \cite{Lie91}.

  \begin{proposition}
    Let $G$ satisfy structure conditions \eqref{I:gen-str} in a cylinder $Q_{\rho}:=K_{\rho}\times [t_0,t_1]$, and let $\zeta$ be a cutoff function on the cylinder $Q_\rho$, vanishing on the parabolic boundary of $Q_\rho$ with $0\le\zeta\le1$.  Define constants $r$, $s$, and $q$ by
    \begin{equation}\label{S4:LocalE-rsq}
    r=1- \frac{1}{g_1}, \quad s=\frac{g_0}{g_1}, \quad \text{and} \quad q= 2g_1.
    \end{equation}
\begin {enumerate}
 \item [(a)] If $u$ is a locally bounded weak supersolution of \eqref{I:gen}, then there exist constants $c_0$, $c_1$, and $c_2$ depending on data such that
    \begin{equation}\label{S4:LocalE}\begin{split}
    &\int_{K_{\rho}\times \{t_1\}} G^{r-1} \left(\frac{\zeta (u-k)_{-}}{\rho}\right) (u-k)_{-}^{s+2} \zeta^{q} \,dx \\
    &\quad + c_0\iint_{Q_{\rho}} G\left(|D(u-k)_{-}|\right) G^{r-1} \left(\frac{\zeta (u-k)_{-}}{\rho}\right) (u-k)_{-}^{s} \zeta^{q} \,dx\,dt \\
    &\leq c_1\iint_{Q_{\rho}} G^{r-1} \left(\frac{\zeta (u-k)_{-}}{\rho}\right) (u-k)_{-}^{s+2} \zeta^{q-1}\left| \zeta_{t}\right| \,dx\,dt \\
    &\quad + c_2 \iint_{Q_{\rho}} G\left(|D\zeta| \zeta(u-k)_{-}\right) G^{r-1} \left(\frac{\zeta (u-k)_{-}}{\rho}\right) (u-k)_{-}^{s}  \,dx\,dt
    \end{split}\end{equation}
for any constant $k$.
\item [(b)] If $u$ is a locally bounded weak subsolution of \eqref{I:gen}, then there exist constants $c_0$, $c_1$, and $c_2$ depending on data such that
    \begin{equation}\label{S4:LocalE+}\begin{split}
    &\int_{K_{\rho}\times \{t_1\}} G^{r-1} \left(\frac{\zeta (u-k)_{+}}{\rho}\right) (u-k)_{+}^{s+2} \zeta^{q} \,dx \\
    &\quad + c_0\iint_{Q_{\rho}} G\left(|D(u-k)_{+}|\right) G^{r-1} \left(\frac{\zeta (u-k)_{+}}{\rho}\right) (u-k)_{+}^{s} \zeta^{q} \,dx\,dt \\
    &\leq c_1\iint_{Q_{\rho}} G^{r-1} \left(\frac{\zeta (u-k)_{+}}{\rho}\right) (u-k)_{+}^{s+2} \zeta^{q-1}\left| \zeta_{t}\right| \,dx\,dt \\
    &\quad + c_2 \iint_{Q_{\rho}} G\left(|D\zeta| \zeta(u-k)_{+}\right) G^{r-1} \left(\frac{\zeta (u-k)_{+}}{\rho}\right) (u-k)_{+}^{s}  \,dx\,dt
    \end{split}\end{equation}
for any constant $k$.
\end {enumerate}
    \end{proposition}

    \begin{proof}
To prove (a), we assume that $u$ is differentiable in terms of the time variable. Such an assumption is removed by applying Steklov average.

    The choices \eqref{S4:LocalE-rsq} are made to satisfy
    \begin{subequations} \label {S4:1-3}
    \begin{gather}
(r-1)g_1 + (s+1) > 0, \label{S4:1} \\
    (r-1)g_0 + s \leq 0,  \label{S4:2} \\
    (r-1)g_1 + q \ge 0, \label{S4:3}
    \end{gather}
    \end{subequations}
Inequality  \eqref{S4:1} implies that $G^{r-1}(\sigma)\sigma^{s+1}$ is increasing with respect to $\sigma$, and inequality  
\eqref {S4:2} implies that $G^{r-1}(\sigma)\sigma^{s}$ is nonincreasing with respect to $\sigma$.

 We use the test function
    \[
    \varphi(x,t) = G^{r-1}\left(\frac{\zeta (u-k)_{-}}{\rho}\right)(u-k)_{-}^{s+1}\zeta^{q},
    \]
in the integral inequality
\begin {equation} \label {S4:supersolution}
\iint _{Q_\rho} u_t\varphi\, dx\,dt + \iint_{Q_\rho} D\varphi\cdot \textbf{A}\,dx,dt \ge0.
\end {equation}

    For simpler notation, let $\bar{u}:= (u-k)_{-}$. Then we have
    \begin{align*}
    D\varphi
    &= \left\{(r-1) \frac{\zeta \bar{u}}{\rho}g\left(\frac{\zeta \bar{u}}{\rho}\right) + (s+1) G\left(\frac{\zeta \bar{u}}{\rho}\right)\right\} G^{r-2} \left(\frac{\zeta \bar{u}}{\rho}\right)\bar{u}^{s}\zeta^{q}D\bar{u} \\
    &\quad + \left\{(r-1)\frac{\zeta \bar{u}}{\rho}g\left(\frac{\zeta \bar{u}}{\rho}\right)  + qG\left(\frac{\zeta \bar{u}}{\rho}\right)\right\}G^{r-2}\left(\frac{\zeta \bar{u}}{\rho}\right) \bar{u}^{s+1}\zeta^{q-1} D\zeta .
    \end{align*}
From the second inequality of \eqref {I:DeltaNabla} and the definition of $r$, it follows that
    \begin {align*}
    (r-1) \frac{\zeta \bar{u}}{\rho}g\left(\frac{\zeta \bar{u}}{\rho}\right) + (s+1) G\left(\frac{\zeta \bar{u}}{\rho}\right)
    &\ge[ (r-1)g_1+ (s+1)]G\left(\frac{\zeta \bar{u}}{\rho}\right) \\
    &= sG\left(\frac{\zeta \bar{u}}{\rho}\right).
    \end {align*}
In addition, the second inequality of \eqref {I:DeltaNabla} and \eqref {S4:3}  imply that
\begin {align*}
(r-1)\frac{\zeta \bar{u}}{\rho}g\left(\frac{\zeta \bar{u}}{\rho}\right)  + qG\left(\frac{\zeta \bar{u}}{\rho}\right) 
 &\ge[ (r-1)g_1+q]G\left(\frac{\zeta \bar{u}}{\rho}\right)  \\
 &\ge0.
 \end {align*}
It then follows from the first inequality of \eqref {I:DeltaNabla} that
\begin {align*}
\left| (r-1)\frac{\zeta \bar{u}}{\rho}g\left(\frac{\zeta \bar{u}}{\rho}\right)  + qG\left(\frac{\zeta \bar{u}}{\rho}\right) \right|
&= (r-1)\frac{\zeta \bar{u}}{\rho}g\left(\frac{\zeta \bar{u}}{\rho}\right)  + qG\left(\frac{\zeta \bar{u}}{\rho}\right) \\
&\le [(r-1)g_0+q]G\left(\frac{\zeta \bar{u}}{\rho}\right).
\end {align*}
Hence
    \begin{equation}\label{S4:LocalE:space01}\begin{split}
    & \quad \iint_{Q_{\rho}}  \textbf{A}(x,t,u,Du) \cdot D\varphi \,dx\,dt \\
    & \leq -s C_0 \iint_{Q_{\rho}} G(|Du|) G^{r-1}\left(\frac{\zeta \bar{u}}{\rho}\right)\bar{u}^{s}\zeta^{q} \,dx\,dt \\
    & \quad + \{(r-1)g_0 + q\} C_1 \iint_{Q_{\rho}} g(|Du|)|D\zeta|G^{r-1}\left(\frac{\zeta \bar{u}}{\rho}\right)\bar{u}^{s+1}\zeta^{q-1} \,dx\,dt .
    \end{split}\end{equation}
In Lemma~\ref{S1:ineq}(e), set $\sigma_1 =|D\zeta|\bar{u}/\zeta$ and $\sigma_2 =|Du|$ to obtain, for any $\epsilon_1 >0$, that
    \begin{equation}\label{S4:LocalE:young}\begin{split}
    & G^{r-1}\left(\frac{\zeta \bar{u}}{\rho}\right)\bar{u}^{s}\zeta^{q} g(|Du|)\frac{|D\zeta \bar{u}|}{\zeta} \\
    &\qquad\leq  \epsilon_1  g_1 G(|Du|) G^{r-1}\left(\frac{\zeta \bar{u}}{\rho}\right)\bar{u}^{s}\zeta^{q} \\
    &\qquad + \epsilon_{1}^{1-g_1}g_1 G\left(\frac{|D\zeta|\bar{u}}{\zeta}\right)G^{r-1}\left(\frac{\zeta \bar{u}}{\rho}\right)\bar{u}^{s}\zeta^{q}.
    \end{split}\end{equation}
In particular, if we choose
\[
\epsilon_1= \frac {sC_0}{2g_1[(r-1)g_0+q]C_1}
\]
and if we use Lemma~\ref {S1:ineq} to estimate
\[
G\left(\frac{|D\zeta|\bar{u}}{\zeta}\right) \le \zeta^{-2g_1}G\left(|D\zeta|\zeta\bar{u}\right),
\]
we infer that
\begin {subequations}
\begin {gather}
\iint_{Q_\rho} D\varphi\cdot \textbf{A}\,dx\,dt \le -\frac 12c_0I_0+\frac 12c_2I_2, \label {S4:DphiA}
\intertext {with}
c_0= sC_0, \label {S4:c0} \\
c_2=2\epsilon_1^{1-g_1}g_1[(r-1)g_0+q]C_1, \label {S4:c2} \\
I_0= \iint_{Q_{\rho}} G(|D\bar u|) G^{r-1} \left(\frac{\zeta \bar u}{\rho}\right) \bar u^{s} \zeta^{q} \,dx\,dt, \\
I_2= \iint_{Q_{\rho}} G\left(|D\zeta| \zeta\bar u\right) G^{r-1} \left(\frac{\zeta \bar u}{\rho}\right) \bar u^{s}\zeta^{q-2g_1}  \,dx\,dt
\end {gather}
\end {subequations}

    Now, by setting
    \[
    F = \int_{0}^{\bar{u}} G^{r-1}\left(\frac{\zeta \alpha}{\rho}\right)\alpha^{s+1} \,d\alpha ,
    \]
we infer that
    \begin{equation}\label{S4:LocalE:time01}
 \iint_{Q_{\rho}} u_t \varphi(x,t) \,dx\,dt 
 = \left. -\int_{K_{\rho}\times \{t\}} F \zeta^{q} dx \right|_{t_0}^{t_1} + q \iint_{Q_{\rho}} F \zeta^{q-1}\zeta_{t} \,dx\,dt.
\end{equation}
    
We now note that $F\ge 0$ and that, because $G^{r-1}(\sigma)\sigma^s$ is increasing with respect to $\sigma$,
\[
F\le G^{r-1}\left( \frac {\zeta\bar u}\rho \right)\bar u^{s+2}.
\]    
Hence
\begin {equation} \label {S4:Ft}
\iint_{Q_{\rho}} F \zeta^{q-1}\zeta_{t} \,dx\,dt \le \iint_{Q_\rho} G^{r-1}\left( \frac {\zeta\bar u}\rho\right)\bar u^{s+2}\zeta^{q-1}\left|\zeta_t\right|\,dx\,dt.
\end {equation}
In addition, because $G^{r-1}(\sigma)\sigma^{s-1}$ is decreasing with respect to $\sigma$, we infer that
\[ 
F \ge G^{r-1}\left( \frac {\zeta\bar u}\rho\right)\bar u^{s-1} \int_0^{\bar u}\alpha\,d\alpha 
= \frac 12G^{r-1}\left( \frac {\zeta\bar u}\rho\right)\bar u^{s+1},
\] 
Therefore
\begin {equation} \label {S4:Ft0}
\left. \int_{K_{\rho}\times\{t\}} F \zeta^{q} dx \right|_{t_0}^{t_1} \le
- \frac 12\int_{K_\rho\times\{t_1\}}G^{r-1}\left( \frac {\zeta\bar u}\rho\right)\bar u^{s+1}\, dx.
 \end {equation}
 The proof is completed, with $c_0$ and $c_2$ given by \eqref {S4:c0} and \eqref {S4:c2} and $c_1=2q$, by combining \eqref  {S4:supersolution}, \eqref {S4:DphiA}, \eqref {S4:LocalE:time01}, \eqref {S4:Ft}, and \eqref {S4:Ft0}.
 
 The proof of (b) is essentially the same with $(u-k)_+$ in place of $(u-k)_-$.
\end{proof}
    
Note that, if we assume \eqref {S4:1-3} and the inequality $q\ge 2g_1$ in place of \eqref {S4:LocalE-rsq},  then \eqref {S4:LocalE} for supersolutions (or \eqref {S4:LocalE+} for subsolutions) holds with the constants determined also by $r$, $s$, and $q$.  We have made the choices in \eqref {S4:LocalE-rsq} for convenience only.

    \subsection{The logarithmic energy estimate}

    The logarithmic energy estimate \eqref{S4:LogE} which is used to prove Proposition~\ref{S3:prop2} is modified from the one in Section II.3-(ii) from \cite{DB93} and similar to the logarithmic estimate in Section 3.3 from \cite{Urb08}. The functions $h$ and $H$ are defined in Lemma~\ref{S1:H-ineq}.

    \begin{proposition}
   Assume that $G$ satisfies \eqref{I:gen-str} in a cylinder $K_{R}\times [t_0, t_1]$. Let $q\ge g_1$ and $\delta\in (0,1)$ be constants, and let $\zeta$ be a cut-off function which is independent of the time variable.
\begin {enumerate}
\item [(a)]   
    Let $u$ be a nonnegative weak supersolution of \eqref{I:gen}and let $k$ be a positive constant.  Then 
    \begin{equation}\label{S4:LogE}\begin{split}
     \int_{K_{R}\times \{t_1\}} H(\Psi^2) \zeta^{q} \,dx  &+ C_0 (4 g_0 - 2) \int_{t_0}^{t_1}\int_{K_{R}} G(|Du|) h(\Psi^2) \left(\Psi'\right)^{2} \zeta^{q}\,dx\,dt \\
      &\quad\leq \int_{K_{R}\times \{t_0\}} H(\Psi^2) \zeta^{q} \,dx \\
    &{} +C^* \int_{t_0}^{t_1}\int_{K_{R}} h(\Psi^2) \Psi \left(\Psi'\right)^2 G\left(\frac{|D\zeta|}{|\Psi'|}\right) \zeta^{q-g_1}\,dx\,dt
    \end{split}\end{equation}
    where 
\[
C^*=\frac{C_0}{g_1} \left(\frac{2 q g_1 C_1}{C_0}\right)^{g_1}
\]
and
    \[
    \Psi(u) = \ln^{+} \left[\frac{k}{(1+ \delta) k - (u-k)_{-}}\right].
    \]
\item [(b)] If $u$ is a nonpositive weak subsolution of \eqref {I:gen} and $k$ is a negative constant, then \eqref {S4:LogE} holds with
    \[
    \Psi(u) = \ln^{+} \left[\frac{k}{(1+ \delta) k + (u-k)_{+}}\right].
    \]
\end {enumerate}  
    \end{proposition}

    \begin{proof}

As before, to prove (a), we assume that $u$ is differentiable in terms of the time variable and later such an assumption is removed by applying the Steklov average.

Define the test function
    \[
    \varphi = 2 h(\Psi^{2})\Psi \Psi' \zeta^{q}.
    \]
and note that
    \begin{align*}
    \Psi' (u) &= \frac{-1}{(1+\delta) k - (u-k)_{-}},\\ 
    \Psi'' (u) &= \frac{1}{\left[ (1+ \delta) k - (u-k)_{-} \right]^{2}} = (\Psi')^{2}.
    \end{align*}
Since $u$ is nonnegative, it follows that $0 \leq (u-k)_{-} \leq k$, and therefore 
    \[
    \frac{1}{(1+\delta)k} \leq \left| \Psi' \right| \leq \frac{1}{\delta k}, \quad 0\leq \Psi \leq \ln^{+} \frac{1}{\delta} .
    \]
    Moreover, $\varphi \in L^{\infty}$ and $D\varphi \in L^{\infty}$, so $\varphi$  is an admissible test function.

    First, we have
    \begin{align*}
    &\int_{t_0}^{t_1}\int_{K_{R}} u_{t} 2 h\left( \Psi^2 \right)\Psi \Psi' \zeta^{q} \,dx \,dt \\
    &= \int_{t_0}^{t_1}\int_{K_{R}} \left[\frac{d}{dt}H\left( \Psi^2 \right)\right]\zeta^{q} \,dx \,dt \\
    &= \int_{K_{R}\times \{t_1\}} H\left( \Psi^2\right)\zeta^{q} \,dx - \int_{K_{R}\times \{t_0\}} H\left( \Psi^2\right)\zeta^{q} \,dx.
    \end{align*}

   Second, we take the derivative of the test function:
    \begin{equation*}\begin{split}
    D\varphi &= [4h'(\Psi^2)(\Psi\Psi')^2 +2h(\Psi^2)(\Psi')^2+2h(\Psi^2)\Psi\Psi'']
    \zeta^{q}Du  \\
             &\quad +  2q h(\Psi^2) \Psi \Psi' \zeta^{q-1}D\zeta.
    \end{split}\end{equation*}
    Using the inequality Lemma~\ref{S1:H-ineq} (c) and the observation that $\Psi''=(\Psi')^2$, we estimate
    \[
    4h'(\Psi^2)(\Psi\Psi')^2 +2h(\Psi^2)(\Psi')^2+2h(\Psi^2)\Psi\Psi'' \ge
    [4g_0-2+2\Psi]h(\Psi^2)(\Psi')^2.
    \]
    It follows that
    \begin{equation}\label{S4:LogE:space}\begin{split}
    &\int_{t_0}^{t_1}\int_{K_{R}} \textbf{A}(x,t,u,Du)\cdot D\varphi \,dx \,dt \\
    &\geq C_0 \int_{t_0}^{t_1}\int_{K_{R}} G(|Du|)h(\Psi^2)[4g_0 - 2 + \Psi] \left(\Psi'\right)^2 \zeta^{q} \,dx \,dt \\
    &\quad - 2q C_1 \int_{t_0}^{t_1}\int_{K_{R}} g(|Du|) h(\Psi^2)\Psi \left|\Psi'\right| \zeta^{q-1} |D\zeta| \,dx \,dt.
    \end{split}\end{equation}
    Using Lemma~\ref{S1:ineq}(e), we infer that 
    \[\begin{split}
    & g(|Du|)h(\Psi^2)\Psi \left|\Psi'\right| \zeta^{q-1}\left|D\zeta\right| \\
    &=  h(\Psi^2)\Psi\left(\Psi'\right)^2 \zeta^{q} g(|Du|) \frac{|D\zeta|}{\zeta \left| \Psi' \right|} \\
    &\leq h(\Psi^2)\Psi\left(\Psi'\right)^2 \zeta^{q}\left[ \epsilon_{2} g_1 G(|Du|) + \epsilon_{2}^{1-g_1} g_1 G\left( \frac{|D\zeta|}{ \zeta \left|\Psi' \right|} \right)\right]
    \end{split}\]
    for any $\epsilon_{2} > 0$. Choosing $\epsilon_2= C_0/(2qg_1C_1)$  leads to \eqref{S4:LogE}.
    
Again the proof of (b) is similar.
    \end{proof}

    \subsection{Poincar{\'e} type inequalities}
We start by quoting Proposition 2.1 from Chapter I of \cite {DB93}.
\begin {lemma} \label {Lpoin}
Let $\Omega$ be a bounded convex subset of $\mathbb R^N$ and let $\varphi$ be a nonnegative continuous function on $\overline {\Omega}$ such that $\varphi\le1$ in $\Omega$ and such that the sets $\{x\in\Omega: \varphi(x)>k\}$ are convex for all $k\in(0,1)$. Then, for any $p\ge1$, there is a constant $C$ determined only by $N$ and $p$ such that
\begin {equation} \label {Lpoin:E}
\left( \int_\Omega \varphi|v|^p\, dx\right)^{1/p} \le C\frac {(\diam \Omega)^N}{|\{x\in \Omega:v(x)=0,\ \varphi(x)=1\}|^{(N-1)/N}} \left( \int_\Omega \varphi|Dv|^p\,dx \right)^{1/p}
\end {equation}
for all $v\in W^{1,p}$.
\end {lemma}
Note that if the set $\{x\in \Omega:v(x)=0,\ \varphi(x)=1\}$ has measure zero, then \eqref {Lpoin:E} is true because the right hand side is infinite.

   \begin{corollary}\label{S4:Poincare-cor}
    Let $v \in W^{1,1}\left(K_{\rho}^{x_0}\right) \cap C\left(K_{\rho}^{x_0}\right)$ for some $\rho>0$ and some $x_0 \in \mathbb{R}^{N}$ and let $k$ and $l$ be any pair of real numbers such that $k<l$. Then there exists a constant $\gamma$ depending only upon $N,p$ and independent of $k,l,v,x_0,\rho$, such that
    \[\begin{split}
    & (l-k)\left| K_{\rho}^{x_0}\cap \{v > l \} \right|\\
    &\leq \gamma \frac{\rho^{N+1}}{\left| K_{\rho}^{x_0}\cap \{v \le k\}\right|}
    \int_{ K_{\rho}^{x_0}\cap \{: k < v < l\}} |Dv| \,dx .
    \end{split}\]
    \end{corollary}
    This lemma is inequality (5.5) from Chapter 2 in \cite {LaSoUr67} (see also  page 5 from \cite{DB93}).

    \subsection{Embedding theorem}

Our next result is a variation on the Sobolev imbedding theorem.

    \begin{theorem}\label{S4:Embedding-theorem}
    For a nonnegative function $v\in W_{0}^{1,1}(Q)$ where $Q=K\times[t_0, t_1]$, $K\subset \mathbb{R}^{N}$, we have
    \begin{equation}\label{EB00}\begin{split}
    \iint_{Q} v \,dx\,dt &\leq C(N) |{Q} \cap \{v>0\}|^{\frac{1}{N+1}} \times \\
    &\quad \left[\esssup_{t_0 \leq t \leq t_1} \int_{K} v \,dx\right]^{\frac{1}{N+1}} \cdot \left[\iint_{Q} |Dv| \,dx\,dt \right]^{\frac{N}{N+1}}.
    \end{split}\end{equation}
    \end{theorem}

    \begin{proof}
    First, by H\"{o}lder's inequality, we obtain
    \begin{equation}\label{EB01}
    \iint_{Q} v \,dx\,dt \leq |Q \cap\{v>0\}|^{\frac{1}{N+1}} \cdot \left[\iint_{Q} v^{\frac{N+1}{N}} \,dx\,dt\right]^{\frac{N}{N+1}}.
    \end{equation}
    Second, by H\"{o}lder's inequality and Sobolev's inequality for $p=1$, we have
    \begin{equation}\label{EB02}\begin{split}
    \int_{K} v^{\frac{N+1}{N}} \,dx
    &\leq \left[\int_{K} v^{\frac{N}{N-1}} \,dx\right]^{\frac{N-1}{N}} \cdot \left[\int_{K} v \,dx\right]^{\frac{1}{N}} \\
    &\leq C(N) \int_{K} |Dv| \,dx \cdot \left[\int_{K} v \,dx\right]^{\frac{1}{N}}.
    \end{split}\end{equation}

    Combining two inequalities \eqref{EB01} and \eqref{EB02} produces the inequality \eqref{EB00}.
    \end{proof}

    \subsection{Iteration}

Finally, we recall Lemma 4.1 from Chapter I in \cite{DB93}.
    \begin{lemma}\label{S4:Iteration-lemma}
    Let $\{Y_n\},$ $n=0,1,2,\ldots,$ be a sequence of positive numbers, satisfying the recursive inequalities
    \begin {equation} \label {S4:Iteration-lemma:E}
    Y_{n+1} \leq C b^{n} Y_{n}^{1+\alpha}
    \end {equation}
    where $C,b >1$ and $\alpha >0$ are given numbers. If
    \[
    Y_0 \leq C^{-\frac{1}{\alpha}}b^{-\frac{1}{\alpha^2}},
    \]
    then $\{Y_n\}$ converges to zero as $n\rightarrow \infty$.
    \end{lemma}



\begin{bibdiv}
   \begin{biblist}



   \bib{Chen84}{article}{
   author={Chen, Y. Z.},
   title={H\"older estimates for solutions of uniformly degenerate
   quasilinear parabolic equations},
   note={A Chinese summary appears in Chinese Ann. Math. Ser. A {\bf 5}
   (1984), no. 5, 663},
   journal={Chinese Ann. Math. Ser. B},
   volume={5},
   date={1984},
   number={4},
   pages={661--678},
    }


   \bib{ChDB88}{article}{
   author={Chen, Y. Z.},
   author={DiBenedetto, E.},
   title={On the local behavior of solutions of singular parabolic
   equations},
   journal={Arch. Rational Mech. Anal.},
   volume={103},
   date={1988},
   number={4},
   pages={319--345},
    }

   \bib{ChDB92}{article}{
   author={Chen, Y. Z.},
   author={DiBenedetto, E.},
   title={H\"older estimates of solutions of singular parabolic equations
   with measurable coefficients},
   journal={Arch. Rational Mech. Anal.},
   volume={118},
   date={1992},
   number={3},
   pages={257--271},
  }
		

   \bib{DG57}{article}{
   author={De Giorgi, E.},
   title={Sulla differenziabilit\`a e l'analiticit\`a delle estremali degli
   integrali multipli regolari},
   language={Italian},
   journal={Mem. Accad. Sci. Torino. Cl. Sci. Fis. Mat. Nat. (3)},
   volume={3},
   date={1957},
   pages={25--43},
    }
		
   \bib{DB83}{article}{
   author={DiBenedetto, E.},
   title={$C^{1+\alpha }$ local regularity of weak solutions of degenerate
   elliptic equations},
   journal={Nonlinear Anal.},
   volume={7},
   date={1983},
   number={8},
   pages={827--850},
    }

\bib{DB86j}{article}{
   author={DiBenedetto, E.},
   title={A boundary modulus of continuity for a class of singular parabolic equations},
   journal={J. Differential Equations},
   volume={6},
   date={1986},
   number={3},
   pages={418--447},
    }

   \bib{DB86}{article}{
   author={DiBenedetto, E.},
   title={On the local behaviour of solutions of degenerate parabolic
   equations with measurable coefficients},
   journal={Ann. Scuola Norm. Sup. Pisa Cl. Sci. (4)},
   volume={13},
   date={1986},
   number={3},
   pages={487--535},
    }


   \bib{DB88}{article}{
   author={DiBenedetto, E.},
   title={Intrinsic Harnack type inequalities for solutions of certain
   degenerate parabolic equations},
   journal={Arch. Rational Mech. Anal.},
   volume={100},
   date={1988},
   number={2},
   pages={129--147},
    }

   \bib{DB93}{book}{
   author={DiBenedetto, E.},
   title={Degenerate parabolic equations},
   series={Universitext},
   publisher={Springer-Verlag},
   place={New York},
   date={1993},
   pages={xvi+387},
    }


   \bib{DBGiVe06}{article}{
   author={DiBenedetto, E.},
   author={Gianazza, U.},
   author={Vespri, V.},
   title={Local clustering of the non-zero set of functions in $W^{1,1}(E)$},
   journal={Atti Accad. Naz. Lincei Cl. Sci. Fis. Mat. Natur. Rend. Lincei
   (9) Mat. Appl.},
   volume={17},
   date={2006},
   number={3},
   pages={223--225},
    }
		

   \bib{DBGiVe08}{article}{
   author={DiBenedetto, E.},
   author={Gianazza, U.},
   author={Vespri, V.},
   title={Harnack estimates for quasi-linear degenerate parabolic
   differential equations},
   journal={Acta Math.},
   volume={200},
   date={2008},
   number={2},
   pages={181--209},
    }
		

   \bib{DBGiVe10}{article}{
   author={DiBenedetto, E.},
   author={Gianazza, U.},
   author={Vespri, V.},
   title={A new approach to the expansion of positivity set of non-negative
   solutions to certain singular parabolic partial differential equations},
   journal={Proc. Amer. Math. Soc.},
   volume={138},
   date={2010},
   number={10},
   pages={3521--3529},
    }

\bib{DBGiVe12}{book}{
   author={DiBenedetto, E.},
   author={Gianazza, U.},
   author={Vespri, V.},
   title={Harnack's inequality for degenerate and singular parabolic
   equations},
   series={Springer Monographs in Mathematics},
   publisher={Springer, New York},
   date={2012},
   pages={xiv+278},
}

     \bib{GiSuVe10}{article}{
   author={Gianazza, U.},
   author={Surnachev, M.},
   author={Vespri, V.},
   title={A new proof of the H\"older continuity of solutions to $p$-Laplace
   type parabolic equations},
   journal={Adv. Calc. Var.},
   volume={3},
   date={2010},
   number={3},
   pages={263--278},
    }

		

   \bib{Hong95}{article}{
   author={Hongjun, Y.},
   author={Songzhe, L.},
   author={Wenjie, G.},
   author={Xiaojing, X.},
   author={Chunling, C.},
   title={Extinction and positivity for the evolution $p$-Laplacian equation
   in ${\bf R}^n$},
   journal={Nonlinear Anal.},
   volume={60},
   date={2005},
   number={6},
   pages={1085--1091},
    }

 \bib{Hwang}{book}{
   author={Hwang, S.},
   title={H\"{o}lder regularity of solutions of generalised $p-$Laplacian type parabolic equations},
   publisher={Graduate Theses and Dissertations},
   place={Iowa State University},
   date={2012},
   note={Paper 12667. http://lib.dr.iastate.edu/etd/12667},
   }



   \bib{KrRu61}{book}{
   author={Krasnosel{\cprime}ski{\u\i}, M. A.},
   author={Ruticki{\u\i}, Ja. B.},
   title={Convex functions and Orlicz spaces},
   series={Translated from the first Russian edition by Leo F. Boron},
   publisher={P. Noordhoff Ltd.},
   place={Groningen},
   date={1961},
   pages={xi+249},
    }

    \bib{LaUr61}{article}{
   author={Lady{\v{z}}enskaja, O. A.},
   author={Ural{\cprime}ceva, N. N.},
   title={Quasilinear elliptic equations and variational problems in several
   independent variables},
   language={Russian},
   journal={Uspehi Mat. Nauk},
   volume={16},
   date={1961},
   number={1 (97)},
   pages={19--90},
}


    \bib{LaSoUr67}{book}{
   author={Lady{\v{z}}enskaja, O. A.},
   author={Solonnikov, V. A.},
   author={Ural{\cprime}ceva, N. N.},
   title={Linear and quasilinear equations of parabolic type},
   language={Russian},
   series={Translated from the Russian by S. Smith. Translations of
   Mathematical Monographs, Vol. 23},
   publisher={American Mathematical Society},
   place={Providence, R.I.},
   date={1967},
   pages={xi+648},
    }

    \bib{LaUr68}{book}{
   author={Ladyzhenskaya, O. A.},
   author={Ural{\cprime}tseva, N. N.},
   title={Linear and quasilinear elliptic equations},
   series={Translated from the Russian by Scripta Technica, Inc. Translation
   editor: Leon Ehrenpreis},
   publisher={Academic Press},
   place={New York},
   date={1968},
   pages={xviii+495},
    }


    \bib{Lie91}{article}{
   author={Lieberman, G. M.},
   title={The natural generalization of the natural conditions of
   Ladyzhenskaya and Ural\cprime tseva for elliptic equations},
   journal={Comm. Partial Differential Equations},
   volume={16},
   date={1991},
   number={2-3},
   pages={311--361},
    }


    \bib{Lie96}{book}{
   author={Lieberman, G. M.},
   title={Second order parabolic differential equations},
   publisher={World Scientific Publishing Co. Inc.},
   place={River Edge, NJ},
   date={1996},
   pages={xii+439},
    }

    \bib{Lie04}{article}{
author={Lieberman, G. M.},
title={Orlicz spaces, $\alpha$-decreasing functions, and the $\Delta_2$ condition},
   journal={Colloq. Math.},
   volume={101},
   date={2004},
   number={1},
   pages={113--120},
}
  
    \bib{Mo61}{article}{
   author={Moser, J. },
   title={A new proof of De Giorgi's theorem concerning the regularity
   problem for elliptic differential equations},
   journal={Comm. Pure Appl. Math.},
   volume={13},
   date={1960},
   pages={457--468},
    }
	

    \bib{RaRe91}{book}{
   author={Rao, M. M.},
   author={Ren, Z. D.},
   title={Theory of Orlicz spaces},
   series={Monographs and Textbooks in Pure and Applied Mathematics},
   volume={146},
   publisher={Marcel Dekker Inc.},
   place={New York},
   date={1991},
   pages={xii+449},
    }

    \bib{Urb08}{book}{
   author={Urbano, J. M.},
   title={The method of intrinsic scaling},
   series={Lecture Notes in Mathematics},
   volume={1930},
   note={A systematic approach to regularity for degenerate and singular
   PDEs},
   publisher={Springer-Verlag},
   place={Berlin},
   date={2008},
   pages={x+150},
    }
   \end{biblist}
\end{bibdiv}
\end {document}